\documentclass[11pt,oneside]{amsart}
\usepackage{amsthm}
\usepackage{amstext}
\usepackage{amssymb}

\numberwithin{equation}{section}
\numberwithin{figure}{section}
\theoremstyle{plain}
\newtheorem{thm}{\protect\theoremname}
  \theoremstyle{definition}
  \newtheorem{example}[thm]{\protect\examplename}
  \theoremstyle{remark}
  \newtheorem{rem}[thm]{\protect\remarkname}
  \theoremstyle{plain}
  \newtheorem{prop}[thm]{\protect\propositionname}
  \theoremstyle{plain}
  \newtheorem{lem}[thm]{\protect\lemmaname}
  \theoremstyle{plain}
  \newtheorem{cor}[thm]{\protect\corollaryname}

\usepackage{diagrams}

\newcommand{\Cot}{{T}^*}
\newcommand{\id}{\operatorname{id}}
\newcommand{\R}{\mathbb R}
\newcommand{\Z}{\mathbb Z}
\newcommand{\graph}{\operatorname{gr}}

\newcommand{\complex}{{\mathbb C}}

\newcommand{\reals}{{\mathbb R}}

\newcommand{\integers}{{\mathbb Z}}

\newcommand{\im}{{\rm Im \,}}

\renewcommand{\ker}{{\rm Ker\, }}

\newcommand{\del}{\partial}

\newcommand{\arrows}{\,\lower1pt\hbox{$\longrightarrow$}\hskip-.24in\raise2pt
             \hbox{$\longrightarrow$}\,}

\newcommand{\Algdminus}{{\bf Algd^-}}
\newcommand{\Algdplus}{{\bf Algd^+}}
\newcommand{\Gpdminus}{{\bf Gpd^-}}
\newcommand{\Gpdplus}{{\bf Gpd^+}}
\newcommand{\Gpd}{{\bf Gpd}}
\newcommand{\SGpd}{{\bf SGpd}}
\newcommand{\Algd}{{\bf Algd}}
\newcommand{\Lie}{{\bf Lie}}

\makeatother

  \providecommand{\corollaryname}{Corollary}
  \providecommand{\examplename}{Example}
  \providecommand{\lemmaname}{Lemma}
  \providecommand{\propositionname}{Proposition}
  \providecommand{\remarkname}{Remark}
\providecommand{\theoremname}{Theorem}

\begin{document}

\title{integration of Lie algebroid comorphisms}

\author{\textbf{Alberto S. Cattaneo, Benoit Dherin and Alan Weinstein}}

\address{Institut f\"ur Mathematik\\
Universit\"at Z\"urich--Irchel\\
Winterthurerstrasse 190, CH-8057 Z\"urich\\
Switzerland}
\email{alberto.cattaneo@math.uzh.ch}

\address{ICMC-USP\\
Universidade de S\~ao Paulo\\
Avenida Trabalhador S\~ao-carlense, 400, Centro\\
CEP: 13566-590, S\~{a}o Carlos, SP\\
Brazil}
\email{dherin@icmc.usp.br}

\address{Department of Mathematics\\
University of California\\
Berkeley, CA 94720-3840}
\email{alanw@math.berkeley.edu}

\begin{abstract}
We show that the path construction integration of Lie algebroids by Lie groupoids is
an actual equivalence from the category of integrable Lie algebroids and
\textit{complete} Lie algebroid comorphisms to the category of source
$1$-connected Lie groupoids and Lie groupoid comorphisms. This allows
us to construct an actual symplectization functor in Poisson geometry.
We include examples to show that the integrability of comorphisms and 
Poisson maps may not hold in the absence of a completeness assumption.
\end{abstract}
\maketitle
\tableofcontents{}

\section{Introduction}

A classical result in differential geometry is that any finite dimensional
Lie algebra $\mathcal{G}$ can be integrated by a $1$-connected Lie
group $\Sigma(\mathcal{G})$ (conversely, any Lie group endows its
tangent space at the unit with the structure of a Lie algebra). This
bijective correspondence between finite dimensional Lie algebras and
$1$-connected finite dimensional Lie groups is actually the object
component of an integration functor
\[
\Sigma:{\bf {LieAlg}\longrightarrow{\bf {LieGp}}},
\]
which is an equivalence from the category of finite dimensional Lie
algebras to the category of finite dimensional $1$-connected Lie
groups. 

There are several generalizations of finite dimensional Lie algebras:
infinite-dimensional (e.g Banach) Lie algebras, Lie algebroids, Poisson
manifolds, and $L_{\infty}$-algebras, for instance. In each case,
it is natural to ask whether there is a corresponding integration
functor. For this, we need to find out the right notion for the objects
integrating these generalized Lie algebras as well as the right notion
of morphisms between them.

We will be concerned here with Lie algebroids and Poisson manifolds
(see \cite{mackenzie1987} for a general reference), each of whose
have natural integrating objects: Lie groupoids for Lie algebroids
and symplectic groupoids for Poisson manifolds. 

However, unlike finite dimensional Lie algebras, not all Lie algebroids
are integrable by Lie groupoids. In \cite{CF2004}, Crainic and Fernandes
worked out a criterion to select those that are. In this paper, we
will construct an integration functor for the class of \textit{integrable}
Lie algebroids. To allow non-integrable Lie algebroids to be in the
domain of an integration functor, one must consider integration by
microgroupoids (i.e. germs of groupoids) in the spirit of \cite{CDW2011}
or by differentiable stacks as in \cite{TsengZhu}. 

We will focus on an integration functor for Lie algebroids that translates
into an integration functor for Poisson manifolds (also called ``symplectization
functor'' by Fernandes \cite{fernandes2006}). For this, the main
ingredient is to replace the usual notion of morphisms between Lie
groupoids (i.e. smooth functors between the underlying groupoids,
see \cite{HM1990,mackenzie1987}) and their corresponding infinitesimal
version for Lie algebroids by that of \textit{comorphism}s. Lie algebroid
and Lie groupoid comorphisms were introduced by Higgins and Mackenzie
in \cite{HM1993}. (This notion for Lie groupoids had been studied
earlier under a different name and with a different, but equivalent,
definition by Zakrzewski in \cite{zakrzewski1990II} and Stachura
in \cite{Stachura2000}.) Already there, comorphisms were seen as
the ``correct'' notion of morphisms between Lie algebroids to be
used with applications to Poisson geometry in mind. 

From the perspective of Poisson geometry, Lie algebroid and Lie groupoid
morphisms are not very well-suited, since a Poisson map $\phi$ from
$X$ to $Y$ induces a Lie algebroid morphism from $T^{*}X$ to $T^{*}Y$
(that integrates to a symplectic groupoid morphism) only when $\phi$
is a diffeomorphism. This prevents us from constructing an integration
functor whose domain would contain all Poisson maps. On the other
hand, the cotangent map $T^{*}\phi$ to a Poisson map $\phi$ is always
a comorphism from $T^{*}X$ to $T^{*}Y$. 

There are also a number of facts independent of Poisson geometry that
make comorphisms the ``correct'' notion of morphisms between Lie
groupoids. For instance, a comorphism between Lie groupoids naturally
induces a group morphism between the corresponding groups of bisections
as well as a $C^{*}$-algebra morphism between the corresponding convolution
$C^{*}$-algebras (see \cite{Stachura2000}). This gives functors
from the category of Lie groupoids and comorphisms to the category
of groups and to the category of $C^{*}$-algebras (see \cite{Stachura2000}).
Moreover, the graph of a Lie groupoid comorphism is a monoid map in
the ``category'' of differentiable relations between the monoid
objects associated to the multiplication graph of the corresponding
Lie groupoids (see \cite{zakrzewski1990II}). 

On the other hand, in contrast to Lie algebroid morphisms, Lie algebroid
comorphisms do not always integrate to Lie groupoid comorphisms. The 
same holds in Poisson geometry, where completeness of Poisson maps 
insures integrability in terms of (symplectic) comorphisms. 
We will give an example of a non complete Lie algebroid comorphism 
that is also non integrable, and whose dual is a non complete and non integrable
Poisson map. 

Dazord in \cite{da:groupoide} already stated without proof that both \textit{complete}
Lie algebroid comorphisms and Poisson maps always do integrate to comorphisms. 
In \cite{zakrzewski1990II}, Zakrzewski proved that complete Poisson maps are integrable
to what he called ``morphisms of regular D$^*$-algebras," which turn out to be nothing but
symplectic comorphisms. 

More recently, Caseiro and Fernandes in \cite{CF} proved that a complete Poisson
map $\phi$ from integrable Poisson manifolds $X$ to $Y$ always integrates to a natural left action of the
symplectic groupoid $\Sigma(Y)$  on $X$ with moment map $\phi$. This action naturally induces an
embedded lagrangian subgroupoid integrating the graph of the Poisson map. 
Their proof, at contrast with the one of Zakrzewski which uses the method of characteristics,
is readily transposable to complete Lie algebroid comorphisms. 
They use the existence of lifting properties by complete Poisson maps for both admissible paths and their homotopies,
which also holds for complete Lie algebroid comorphisms and makes them resemble
``Serre fibrations'' in topology. 

These lifting properties  (Proposition \ref{prop: hom. lifting prop. I}
and \ref{prop: hom. lifting. prop. II}) will also be central to our main result (Theorem \ref{thm: Dazord}): namely, 
that the path construction of \cite{CF2001,CF2004,CF2003}, which associates a source $1$-connected
Lie groupoid $\Sigma(A)$ to each integrable Lie algebroid $A$, is an actual equivalence 
from the category of integrable Lie algebroids and complete Lie algebroid 
comorphisms to the category of source $1$-connected
Lie groupoids and Lie groupoid comorphisms. 

As a corollary, we obtain that Lie algebroid comorphisms are integrable if and only if they
are complete, which strengthens Dazord's statement. We show that this
implies a corresponding theorem in Poisson geometry, where the path
construction implements an equivalence between the category of integrable
Poisson manifolds and complete Poisson maps and the category of source
$1$-connected symplectic groupoids and symplectic comorphisms. From
this, we may conclude that Poisson maps are integrable if and only
if they are complete, which was already shown by Zakrzewski in the language
of regular D$^*$-algebras.

\subsubsection*{Acknowledgments}

We are grateful to P. Dazord, R. Fernandes, Z. Liu, K. Mackenzie,
I. Moerdijk, and C. Zhu for input on comorphism literature. A.S.C.
acknowledges partial support from SNF Grant 20-131813. B.D. acknowledges
partial support from NWO Grant 613.000.602, FAPESP grants 2010/15069-8
and 2010/19365-0 , as well as the hospitality of the UC Berkeley mathematics
department and S\~{a}o Paulo University ICMC. A.W. acknowledges partial
support from NSF grant DMS-0707137 and the hospitality of the Institut
Math\'ematique de Jussieu.

\section{Morphisms and comorphisms}

\label{sec-comorphisms} Much of what follows in this section may
already be found in \cite{ch-li:comorphisms,da:groupoide,HM1990,HM1993}.
We work in the smooth category. Let $a:A\to X$ and $b:B\to Y$ be
submersions, which we may think of as families of manifolds parametrized
by $X$ and $Y$. 

A map $\phi$ from $X$ to $Y$ together with a map $\Phi$ to $A$
from the pullback $f^{!}B=X\times_{Y}B$ will be called a \textbf{comorphism}
from $a$ to $b$; $\phi$ will be called the \textbf{core} \textbf{map}
of the comorphism. When the families are vector bundles and $\Phi$
is linear on fibres, we call $(\phi,\Phi)$ a vector bundle comorphism.
It induces a dual vector bundle map $\Phi^{*}$ from $a^{*}:A^{*}\rightarrow X$
to $b^{*}:B^{*}\rightarrow Y$ covering $\phi$ and a pullback map
$\Phi^{\dagger}$ to the space $\Gamma(A)$ of sections of $A$ from
$\Gamma(B)$. 

On the other hand, a \textbf{morphism} from $a$ to $b$ is simply
a bundle map, which we also denote by $(\phi,\Phi)$, where the core
map $\phi$ is the base map of the bundle map, and $\Phi$ is a collection
of smooth maps $\Phi_{x}$ from the fibers $A_{x}$ to $B_{\phi(x)}$.
When $a$ and $b$ are vector bundles and $\Phi$ is linear on fibers,
a morphism $(\phi,\Phi)$ is a vector bundle map. 

As observed in \cite{HM1993}, the notions of morphisms and comorphisms
for vector bundles are dual to each other in the sense that $(\phi,\Phi)$
is a comorphism from $a$ to $b$ if and only if $(\phi,\Phi^{*})$
is a morphism from $a^{*}$ to $b^{*}$ (and conversely). 

We now specialize the notion of morphisms and comorphisms to Lie algebroids
and Lie groupoids, and we introduce corresponding Lie functors (see
also \cite{HM1990,HM1993}).

\subsection{Lie algebroids\label{sub:The-Lie-functor}}

If $A$ and $B$ are Lie algebroids, a vector bundle comorphism $(\phi,\Phi)$
is called a \textbf{Lie algebroid comorphism} if $\Phi^{*}$ is a
Poisson map for the natural Lie-Poisson structures on the dual Lie
algebroids. Equivalently, $\Phi^{\dagger}$ is a homomorphism of Lie
algebras, and
\begin{equation}
\phi_{*}\circ\rho_{A}\circ\Phi=\rho_{B},\label{eq:anchor}
\end{equation}
where $\rho_{A}$ and $\rho_{B}$ are the anchor maps of respectively
$A$ and $B$. 

On the other hand, a vector bundle morphism $(\phi,\Phi)$ is called
a \textbf{Lie algebroid morphism} if $\Phi^{*}$ induces a chain map
from Lie algebroid complexes $\Gamma(\wedge^{\bullet}B)$ to $\Gamma(\wedge^{\bullet}A)$
(see \cite{HM1990}). 

We denote by $\Algdplus$ the category of Lie algebroids and Lie algebroid
\textit{morphisms} and $\Algdminus$ the category of Lie algebroids
and Lie algebroid \textit{comorphisms}. 

Observe that the graph of both a Lie algebroid morphism and a Lie
algebroid comorphism is a Lie subalgebroid of the Lie algebroid product
$A\times B$ and that morphisms and comorphisms coincide when the
core map is a diffeomorphism.
 
\begin{example}
\label{exa: Poisson} Let $\phi:X\rightarrow Y$ be a smooth map.
The tangent map $T\phi$ is a Lie algebroid morphism from $TX$ to
$TY$ (seen as algebroids with identity as anchor and the usual Lie
bracket on vector fields), while the cotangent map $T^{*}\phi$ is
a Lie algebroid comorphism from $T^{*}X$ and $T^{*}Y$ (seen as Lie
algebroids with zero anchor and zero bracket). $T\phi$ and $T^{*}\phi$
are both, at the same time, Lie algebroid morphisms and comorphisms
when $\phi$ is a diffeomorphism.
\end{example}

\begin{example}
If $A=TX$ and $B=TY$ carry the usual Lie algebroid structures (as
in the previous example), then $(\phi,\Phi)$ is a comorphism from
$TX$ to $TY$ when $\phi$ is a submersion and $\Phi$ is the horizontal
lift map of a flat Ehresmann connection over the open submanifold
$\phi(X)\subseteq Y$. 
\end{example}

\begin{example}
If $A$ and $B$ are Lie algebras, considered as Lie algebroids over
a point, then a Lie algebroid comorphism from $A$ to $B$ is a Lie
algebra morphism from $B$ to $A$, while a Lie algebroid morphism
from $A$ to $B$ is a Lie algebra morphism from $A$ to $B$. 
\end{example}

\subsection{Lie groupoids}

Now let $G\rightrightarrows X$ and $H\rightrightarrows Y$ be groupoids
with target and source maps $l_{G},r_{G},l_{H}$ and $r_{H}$. A comorphism
$(\phi,\Phi)$ from $r_{G}$ to $r_{H}$ is called \textbf{a comorphism
of groupoids} if
\begin{enumerate}
\item $\Phi$ takes unit elements to unit elements;
\item it is compatible with the target maps in the sense that, for any $(x,h)$
in the pullback $X\times_{Y}H,$ $(\phi\circ l_{G})(\Phi(x,h))=l_{H}(h);$
\item it is multiplicative in the sense that $\Phi(y,h_{1})\Phi(z,h_{2})=\Phi(z,h_{1}h_{2})$
whenever the products are defined; i.e., when $\phi(y)=l_{H}(h_{2}).$
\end{enumerate}
A groupoid comorphism as above may be represented by its graph $\gamma_{(\phi,\Phi)}$,
which is the smooth closed subgroupoid of $G\times H\rightrightarrows X\times Y$
consisting of those pairs $(g,h)$ for which $g=\Phi(r_{G}(g),h)$.
The objects of $\gamma_{(\phi,\Phi)}$ are just the points of the
graph of $\phi$, and the projection to $H$ of the source fibre of
$\gamma_{(\phi,\Phi)}$ over $(g,\phi(g))$ is a diffeomorphism onto
the source fibre of $H$ over $\phi(g)$. These properties characterize
those subgroupoids of $G\times H$ which are the graphs of comorphisms.

\begin{rem}
Zakrzewski in \cite{zakrzewski1990II} introduced the notion of regular
$D^{*}$-algebra and showed that it coincides with that of Lie groupoid,
observing though that their natural morphisms do not correspond to
Lie groupoid morphisms. $D^{*}$-algebra morphisms were further studied
by Stachura in \cite{Stachura2000}, who called them simply ``groupoid
morphisms.'' From Lemma 4.1 in \cite{zakrzewski1990II} and Proposition
2.6 in \cite{Stachura2000}, one sees that $D^{*}$-algebra morphisms
are exactly the Lie groupoid comorphisms introduced later on by Higgins
and Mackenzie in \cite{HM1993}. However, neither Zakrzewski nor Stachura
discussed a corresponding notion for Lie algebroids. 
\end{rem}

A \textbf{morphism of Lie groupoids} is a functor between the underlying
groupoids, whose object and morphism components are smooth. 

We denote by $\Gpdplus$ the category of Lie groupoids and Lie groupoid
\textit{morphisms} and $\Gpdminus$ the category of Lie groupoids
and Lie groupoid \textit{comorphisms}. 

Correspondingly, there are two Lie functors
\[
\Lie:\Gpd^{\pm}\rightarrow\Algd^{\pm},
\]
as defined in \cite{HM1993,mackenzie1987}, which agree on objects
(i.e. they both send a Lie groupoid to its associated Lie algebroid)
but one sends morphisms to morphisms while the other sends comorphisms
to comorphisms. Geometrically though, the morphism component of both
functors can be defined the ``same way,'' using the object component.
Namely, the underlying graph $\gamma_{(\phi,\Phi)}$ of a groupoid
morphism or a groupoid comorphism (which, in both cases, we denote
by $(\phi,\Phi)$) from $G$ to $H$ is itself a groupoid (actually,
a subgroupoid of $G\times H$). Then $\Lie(\gamma_{(\phi,\Phi)})$
is a subalgebroid of $\Lie(G)\times\Lie(H)$, which is the graph of
a Lie algebroid morphism when $(\phi,\Phi)$ is a Lie groupoid morphism
and the graph of a Lie algebroid comorphism when $(\phi,\Phi)$ is
a Lie groupoid comorphism. 

The two Lie functors are essentially the same on Lie algebras, since
morphisms and comorphisms are the same in this case except for arrow
direction.

\subsection{Integrability and completeness}

We say that a Lie algebroid comorphism between integrable Lie algebroids
is \textbf{integrable} if it is in the image of the Lie functor. This
means that the (possibly immersed) Lie subgroupoid integrating the
comorphism graph (which is a Lie subalgebroid) is, at the same time,
a closed embedded Lie subgroupoid \textit{and} a comorphism (the latter
implying the former). 

If $A$ and $B$ are the Lie algebroids of groupoids $G$ and $H$,
then every Lie algebroid comorphism from $A$ to $B$ may be integrated
locally to a groupoid comorphism from $G$ to $H$. In contrast with
Lie algebroid morphisms, which are always integrable to Lie groupoid
morphisms under a simple connectivity assumption (see \cite[Appendix]{MX2000}
for instance), the global situation for Lie algebroid comorphisms
is more complicated. The following example gives a Lie algebroid comorphism
whose graph integrates, as a Lie algebroid, to an embedded Lie subgroupoid
that is not the graph of a Lie groupoid comorphism.

\begin{example}
The inclusion $i$ of an open subset $X$ in a manifold $Y$ yields
the Lie algebroid comorphism $(i,\id)$ from $TX$ to $TY$ with the
natural Lie algebroid structure. The integration of the Lie subalgebroid
$\gamma_{(i,\id)}$ is the embedded subgroupoid
\[
\Big\{(x,x,x,x):\: x\in X\Big\}\rightrightarrows X\times Y
\]
of the groupoid product $(X\times X)\times(Y\times Y)\rightrightarrows X\times Y$.
This is not the graph of a comorphism, although there are partially
defined maps (namely the identity restricted to $X$) from $\{x\}\times Y$
to $\{x\}\times X$ for each $x\in X$, the union of whose graphs
is the integrating subgroupoid. 
\end{example}

Although ``partially defined'' Lie groupoid comorphisms as in the
example above still compose, and thus form a category, even worse
situations can arise. In general, a Lie algebroid comorphism can be
integrated only to what Dazord calls a ``relation,'' and which we
will call a \textbf{hypercomorphism}. A hypercomorphism from $G\rightrightarrows X$
to $H\rightrightarrows Y$ consists of a map $\phi:X\to Y$ and a
groupoid $R$ over the graph of $\phi$ along with a homomorphism
to $G\times H$ which is an immersion such that the projection to
$H$ is \'{e}tale between source fibres of $R$ and $H$. It is a comorphism
just when these maps between source fibres are diffeomorphisms. The
image of the immersion $R\to G\times H$ is a subgroupoid which can
sometimes be neither smooth nor closed, as we will see in the next
section.

As we will show in Theorem \ref{thm: Dazord}, global integrability
in terms of Lie groupoid comorphisms is guaranteed if the source fibres
of $H$ are 1-connected and the Lie algebroid comorphism is \textbf{complete}.
This means that the pullback map on sections takes complete sections
of $B$ to complete sections of $A$, where a section of a Lie algebroid
is called complete if the anchor maps it to a complete vector field.%
\footnote{If $A$ is integrable to a Lie groupoid $G$, completeness of a section
of $A$ means that the section is the initial derivative of a 1-parameter
group of bisections of $G$.%
} 

This result on global integrability was first announced in \cite{da:groupoide}
without proof. To the best of our knowledge, such a proof has never
appeared, although very close results have been achieved in the context of Poisson geometry for complete Poisson maps 
(which induce complete Lie algebroid comorphisms as we shall see in Section \ref{sec:Application-to-poisson})
by Caseiro and Fernandes in \cite{CF} and by Zakrzewski  in \cite{zakrzewski1990II}.
We will give one in Section \ref{sub:Embeddability} using the path integration techniques developed recently in \cite{CF2001,CF2004}.

We can already see the following: 

\begin{prop}
\label{prop: Lie}Let $G$ and $H$ be Lie groupoids over $X$ and
$Y$ respectively, and let $(\phi,\Phi)$ be a groupoid comorphism
from $G$ to $H$. Then $(\phi,T\Phi)=\Lie(\phi,\Phi)$ is a \textbf{complete}
comorphism from $\Lie(G)$ to $\Lie(H)$, where
\[
(T\Phi)(x,v):=D_{2}\Phi(x,\phi(x))v,
\]
and $D_{2}$ denotes the derivative w.r.t. the second argument. In
other words, integrable Lie algebroid comorphisms are complete. 
\end{prop}

\begin{proof}
To simplify notation, set $A=\Lie(G)$ and $B=\Lie(H)$. Let $s$
be a complete section of $B$. It induces a \textit{complete} left-invariant
vector field $\xi_{H}^{s}$ on the integrating groupoid $H$ (see
\cite[Appendix]{KS1072}), whose corresponding left-invariant flow
$\Psi_{t}^{H}$ exists thus for all $t$. Using $(\phi,\Phi)$, we
define a left-invariant flow on $G$, which also exists for all times:
\[
\bar{\Psi}_{t}^{G}(g)=L_{g}\Big(\Phi\Big(r_{G}(g),\Psi_{t}^{H}(\phi(r_{G}(g))\Big)\Big),
\]
where $L_{g}(g')=gg'$ is the left-translation in $G$. 

On the other hand, the section $(T\Phi)^{\dagger}s$ induces a left-invariant
vector field $\xi_{G}^{(T\Phi)^{\dagger}s}$ on $G$, whose flow $\Psi_{t}^{G}$
projects on the flow $\Psi_{t}^{X}$ on $X$ of $\rho_{A}(T\Phi)^{\dagger}s$,
that is, 
\[
\Psi_{t}^{X}(l_{G}(g))=l_{G}(\Psi_{t}^{G}(g)).
\]
What remains to be proven is that $\Psi_{t}^{G}$ coincides with $\bar{\Psi}_{t}^{G}$:
Since the latter exists for all $t$, this would imply that $\Psi_{t}^{X}$
exists for all $t$ and, thus, that the image $(T\Phi)^{\dagger}s$
of a complete section $s$ by $(\phi,T\Phi)$ is complete. To see
this, let us check that both flows are flows of the same vector field.
Namely, since $\bar{\Psi}_{t}^{G}$ is left-invariant, we have that
\begin{eqnarray*}
\frac{d}{dt}_{\big|t=0}\bar{\Psi}_{t}^{G}(g) & = & DL_{g}(r_{G}(g))D_{2}\Phi(r_{G}(g),\phi(r_{G}(g)))\,\frac{d}{dt}_{\big|t=0}\Psi_{t}^{H}(\phi(r_{G}(g))),\\
 & = & DL_{g}(r_{G}(g))(T\Phi)(r_{G}(g),s(\phi(r_{G}(g))),\\
 & = & DL_{g}(r_{G}(g))((T\phi)^{\dagger}s)(r_{G}(g)),
\end{eqnarray*}
which, by definition, coincides with $\frac{d}{dt}_{\big|t=0}\Psi_{t}^{G}(g)$.
\end{proof}

\begin{example}
Let $(\phi,\Phi)$ be a comorphism between tangent bundle Lie algebroids
$TX$ and $TY$. As noted above, this corresponds to a flat Ehresmann
connection, i.e. a ``horizontal'' foliation of $X$ for which the
projection of each leaf to $Y$ is \'{e}tale. The comorphism is complete
when the connection is complete in the sense that these projections
are all covering maps, i.e. when each path $\sigma:[0,1]\to X$ has
a horizontal lift starting at any point in $\phi^{-1}(\sigma(0))$.

To integrate this comorphism to a hypercomorphism between the fundamental
groupoids $\pi(X)$ and $\pi(Y)$ integrating $TX$ and $TY$ respectively,
we let $R$ be the leafwise fundamental groupoid of the foliation
of $X$. This is a ($1$-connected, but possibly non-Hausdorff) Lie
groupoid over $X$ and may hence be considered as a groupoid over
the graph of $\phi:X\to Y$. An element of $R$ is a homotopy class
of paths with fixed endpoints and contained in a single leaf of the
foliation. Let us call these ``foliated paths''. Mapping each such
class of foliated paths to the homotopy class of paths in $X$ (without
the ``leafwise'' restriction) in which it is contained, and to the
class of the image in $Y$, is a groupoid morphism from $R$ to $G\times H$.
Restricting this morphism to a source fibre of $R$ and projecting
to $\pi(Y)$ takes the homotopy classes of foliated paths beginning
at some $x\in X$ to the homotopy classes of paths in $Y$ beginning
at $\phi(x)$. A neighborhood, in a source fibre of $R$, of the class
of a foliated path $\sigma$ may be identified with a neighborhood
of $\sigma(1)$ in its leaf, and a neighborhood of the class of the
projected path may be identified with a neighborhood of $\phi(\sigma(1))$
in $Y$. The projection from the first neighborhood to the second
is \'{e}tale by the definition of a flat Ehresmann connection, so the
requirements for $R$ to be a hypercomorphism are met. We may describe
the relation $R$ in rough terms by saying that it takes a point $x\in X$
and a path $\rho$ in $Y$ beginning at $\phi(x)$ to its horizontal
lift through $x$. But this horizontal lift may not exist if $(\phi,\Phi)$
is not complete, and it might not be unique since a homotopy of paths
in $Y$ may not have a horizontal lift in the absence of completeness.
\end{example}

\section{An example\label{sec:The-example}}

We give in this section an example of an Ehresmann connection, the
graph of whose integration is neither closed nor embedded.

Let $Y$ be $\reals^{2}$ with cartesian coordinates $(x,y)$ and
polar coordinates $(r,\theta)$. $X$ will be an open subset of $\reals^{2}\times\complex\times\reals$
with polar coordinates $(r,\theta)$ on the first factor, a complex
coordinate $z=Re^{i\Theta}$ on the second, and real coordinate $h$
on the third one. As the notation suggests, $\phi$ will be the projection
on the first factor. $X$ and $Y$ will be $1$-connected, so the
source 1-connected groupoids integrating $TX$ and $TY$ will be $X\times X$
and $Y\times Y$. Since $Y$ is $1$-connected, the leaves of any
foliation defined by a complete Ehresmann connection over $Y$ are
simply connected and therefore have trivial holonomy.

We define $X$ as $\reals^{2}\times\complex\times\reals\setminus J$,
where $J$ is the three-dimensional slab $\{(r,\theta,R,\Theta,h)|r=0,\,-1\leq h\leq1\}$.
Although $J$ is of codimension $2$, the restriction on $h$ leaves
$X$ simply connected, so that its fundamental groupoid is still $X\times X$.
Nevertheless, we can construct an interesting Ehresmann connection
for the submersion $\phi:X\to Y$. For $-1\leq h\leq1$, the horizontal
subspaces of the connection are spanned by the vector fields $\del/\del r$
and $\del/\del\theta+\nu(h)\del/\del\Theta$, where $\nu$ is a smooth
function, not identically zero, supported in the interval $-\frac{1}{2}\leq h\leq\frac{1}{2}.$
This makes sense since $r$ is not zero for these values of $h$.
Outside the support of $\nu$, the vector fields $\del/\del r$ and
$\del/\del\theta$ may be replaced by the cartesian coordinate vector
fields $\del/\del x$ and $\del/\del y$, which extend to the entire
$(x,y)$ plane.

We will think of our Ehresmann connection as a family, parametrized
by $h$, of unitary connections on the trivial complex Lie bundle
over $Y$ with fibre coordinate $z$. The connection form in this
description is $i\nu(h)d\theta$, and the holonomy around a loop encircling
the origin in the $(x,y)$ plane is multiplication by $e^{2\pi i\nu(h)}.$
In the region where $-1\leq h\leq1$ (and so $r$ is not zero), each
leaf lies in a fixed level of $R$ and $h$ and is a covering of the
punctured $(r,\theta)$ plane. The covering is a diffeomorphism if
$R=0$. For positive $R$, the covering has $k$ sheets when $\nu(h)$
has order $k$ as an element of $\reals/\integers$; this includes
the possibilities $k=1$ and $k=\infty$. Over the region where $-1<h<1$,
the leafwise fundamental groupoid may be parametrized by 
\[
\Gamma=(\reals^{+}\times S^{1})\times(\reals^{+}\times\reals)\times\complex\times(-1,1).
\]
 The element 
\[
\gamma=(r,\theta,r',\tau,z,h)
\]
of $\Gamma$ corresponds to the homotopy class of the horizontal path
\[
t\mapsto(r+(r'-r)t,\theta+\tau t,e^{i\nu(h)\tau t}z,h),~0\leq t\leq1.
\]
Thus, the source map is 
\[
(r,\theta,r',\tau,z,h)\mapsto(r,\theta,z,h),
\]
and the target is 
\[
(r,\theta,r',\tau,z,h)\mapsto(r',\theta+\tau,e^{i\nu(h)\tau}z,h).
\]
 The unit elements of the groupoid are defined by the conditions $r=r'$
and $\tau=0$, while the isotropy groups are defined by $r=r'$, $\tau\in2\pi\integers$,
and $\nu(h)\tau\in2\pi\integers$.

We now look at the leafwise fundamental groupoid as the integration
of the Lie algebroid comorphism given by the flat Ehresmann connection.
The Lie algebroid sits inside $TX\times TY$; since $X$ and $Y$
are simply connected, the integrating subgroupoid $S$ should sit
inside $X\times X\times Y\times Y$; it is the image of $\Gamma$
under the target-source map 
\[
(r,\theta,r',\tau,z,h)\mapsto\Big((r,\theta,z,h),(r',\theta+\tau,e^{i\nu(h)\tau}z,h),(r,\theta),(r',\theta+\tau)\Big).
\]
To study the immersion of $\Gamma$ into $X\times X\times Y\times Y$,
we can forget about the last two factors, since they are redundant
(namely, the image of $\Gamma$ lies in the graph of $\phi\times\phi$,
which can be identified with $X\times X$). This image consists of
all $8$-tuples $(r,\theta,z,h,r',\theta',z',h')$ for which there
exists $\tau$ such that $\theta'=\theta+\tau$, $z'=e^{i\nu(h)\tau}z,$
and $h'=h$. The two conditions involving $\tau$ can be combined,
with the elimination of $\tau$, to give $z'=e^{i\nu(h)(\theta'-\theta)}z.$
Where $z$ is nonzero, these define, for each $h$, a hypersurface
in the $4$-torus with coordinates $(\theta,\arg z,\theta',\arg z')$.
The subgroupoid $S\subset X\times X\times Y\times Y$ sits as a family
of these hypersurfaces inside the $8$-dimensional submanifold defined
by $|z'|=|z|$ and $h'=h$, which is a bundle of these $4$-tori over
the space parametrized by $(r,z,h)$.

From this description, we see immediately that $S$ is not closed.
In fact, when $\nu(h)$ is irrational, each of our hypersurfaces is
dense but not closed in its $4$-torus. To see that $S$ has nontrivial
self-intersections, we must look at the section $z=0$ of our complex
line bundle, since otherwise we are simply dealing with flat hypersurfaces
in tori. In fact, when $z=0$, adding an integer multiple of $2\pi$
to $\tau$ does not change the value of the target-source map, but
it does change the image of the derivative as long as $\nu(h)$ is
not an integer. This results in the sought-for nontrivial self-intersections.

\section{Path construction\label{sec:Path-construction}}

In this section, we start by briefly recalling the integration of
Lie algebroids by Lie groupoids in terms of quotients of certain admissible
path sets by homotopies, as in \cite{CF2001,CF2003}. We explain how
this path construction allows us to integrate comorphisms between
Lie algebroids to comorphisms between Lie groupoids. Then we show
that a complete Lie algebroid comorphism from $A$ to $B$ allows
us to lift admissible paths and homotopies in $B$ to $A$ (Proposition
\ref{prop: hom. lifting prop. I} and \ref{prop: hom. lifting. prop. II}).
These lifting properties, which make complete comorphisms resemble
``Serre fibrations,'' will be the main ingredients in the proof
of Theorem \ref{thm: Dazord} in the next section. 

Similar lifting properties in the context of Poisson map integration have already 
been considered by Caseiro and Fernandes in \cite{CF}.

\subsection{Lie algebroid integration}

All the source $1$-connected Lie groupoids integrating an integrable
Lie algebroid $A\rightarrow X$ are isomorphic to the following construction
in terms of homotopy classes of paths \cite{CF2001,CF2003}. Consider
the space $\mathcal{P}(A)$ of admissible paths; i.e., the set of
paths $g:[0,1]\rightarrow A$, $g(t)=(x(t),\eta(t))$, where $x(t)\in X$
and $\eta(t)$ lies in the fiber of $A$ over $x(t)$, such that
\[
\frac{dx(t)}{dt}=\rho(x(t))\eta(t),
\]
where $\rho$ is the anchor map of $A$. The source $1$-connected
Lie groupoid integrating $A$ can be realized as the quotient of $\mathcal{P}(A)$
by a homotopy relation $\thicksim$ that fixes the endpoints of the
base component of the admissible path (see \cite{CF2001,CF2003}).
More precisely, $(x_{1}(t),\eta_{1}(t))$ is homotopic to $(x_{2}(t),\eta_{2}(t))$
iff there is a family 
\[
(x_{1}(t),\eta_{1}(t))\;\overset{s=0}{\longleftarrow}\;(x(t,s),\eta(t,s))\;\overset{s=1}{\longrightarrow}\;(x_{2}(t),\eta_{2}(t))
\]
of admissible paths parametrized by $s\in[0,1]$ that satisfies the
following condition: There exists a section $\beta$ of $A$ defined
along $x(t,s)$ that vanishes for $t=0,1$, such that, locally, 
\begin{eqnarray}
\frac{\partial x^{i}(t,s)}{\partial s} & = & \rho_{a}^{i}(x(t,s))\beta^{a}(t,s),\label{eq: x-comp}\\
\frac{\partial\eta^{c}(t,s)}{\partial s} & = & \frac{\partial\beta^{c}(t,s)}{\partial t}+f_{ab}^{c}(x(t,s))\beta^{a}(t,s)\eta^{b}(t,s),\label{eq: p-comp}
\end{eqnarray}
where $\rho_{a}^{i}(x):U\rightarrow\R$ and $f_{ab}^{c}(x):U\rightarrow\R$
are the structure functions of respectively the anchor map and the
Lie bracket on the sections of $A$ expressed in terms of a system
of trivializing sections $e_{a}:U\rightarrow A_{|U}$ (where $a$
ranges from $1$ to the dimension of the fibers in $A$) over the
local patch with coordinates $x^{i}$. 

We denote by $\Sigma(A)$ the quotient of $\mathcal{P}(A)$ by this
homotopy relation and by $[g]$ the homotopy class of $g$. Since
$A$ is assumed to be integrable, $\Sigma(A)$ is a Lie groupoid over
$X$, whose source and target maps $s,t:\Sigma(A)\rightrightarrows X$
are given by the endpoints of the path projection on the base: $s([g])=x(0)$
and $t([g])=x(1)$. The groupoid product is given by concatenation
of paths $[g][g']=[gg']$, where $g\in[g]$ and $g'\in[g']$ are two
representatives whose ends $t([g])=s([g'])$ match smoothly, and where
\[
(gg')(t)=\left\{ \begin{array}{cc}
2g(2t), & 0\leq t\leq\frac{1}{2},\\
2g'(2t-1), & \frac{1}{2}<t\leq1.
\end{array}\right.
\]
From now on, we will reserve the notation $\Sigma(A)\rightrightarrows X$
for the source $1$-connected Lie groupoid integrating $A$ coming
from the construction above. 

Note that $\Sigma(A)$ exists as a groupoid but not as a manifold
if $A$ is not integrable: for non-integrable Lie algebroids, $\Sigma(A)$
can be realized as a stack (see \cite{TsengZhu}).

\subsection{Comorphism integration\label{sub: path construction}}

Suppose that $A'\rightarrow X'$ is a subalgebroid of an integrable
Lie algebroid $A\rightarrow X$ with integrating $1$-connected Lie
groupoid $\Sigma(A)\rightrightarrows X$. Then $A'$ is automatically
integrable as a Lie algebroid, and we can take its integrating Lie
groupoid to be the one obtained by the path construction; namely,
$\Sigma(A')\rightrightarrows X'$. An admissible path in $A'$ is
by definition also an admissible path in $A$. Moreover, if two admissible
paths are homotopic in $A'$, they also are homotopic in $A$. Therefore,
we have a natural immersion
\[
\iota:\Sigma(A')\rightarrow\Sigma(A),
\]
which is a groupoid morphism. However, this map is general not an
embedding nor is its image a closed submanifold. When the Lie subalgebroid
$A'$ is \textit{over the same base} as $A$, then Moerdijk and Mr\v{c}un
in \cite{MM2006} gave a necessary and sufficient condition for $\iota$
to be a closed embedding.

The situation is similar for the integration of a comorphism $(\phi,\Phi)$
from a Lie algebroid $A\rightarrow X$ to a Lie algebroid $B\rightarrow Y$,
since a comorphism graph $\gamma_{(\phi,\Phi)}$ is a Lie subalgebroid
(over the graph of $\phi$) of the direct product of the Lie algebroids
$A$ and $B$. If $A$ and $B$ are integrable, we can thus realize
the hypercomorphism between the Lie groupoids $\Sigma(A)$ and $\Sigma(B)$
in terms of the path construction as the groupoid immersion
\[
\iota:\Sigma(\gamma_{(\phi,\Phi)})\rightarrow\Sigma(A)\times\Sigma(B).
\]
Section \ref{sec:The-example} gave an explicit example of a comorphism
for which $\iota$ is not an embedding and its image is not a closed
submanifold. 

Let us now describe $\Sigma(\gamma_{(\phi,\Phi)})$ in more explicit
terms. It can be realized as the set of homotopy classes $[\gamma]$
of paths $\gamma(t)=(g(t),h(t))$, where $g(t)$ is an admissible
path in the Lie algebroid $A\rightarrow X$ and $h(t)$ is an admissible
path in the Lie algebroid $B\rightarrow Y$ of the form 
\begin{eqnarray}
g(t) & = & \Big(x(t),\Phi\big(x(t),\xi(t)\big)\Big),\label{eq: g(t)}\\
h(t) & = & \Big(\phi(x(t)),\xi(t)\Big).\label{eq: h(t)}
\end{eqnarray}
In other words, $\gamma$ is an admissible path in the Lie algebroid
$\gamma_{(\phi,\Phi)}\rightarrow\graph\phi$ which satisfies the following
equations
\begin{eqnarray*}
\dot{x}(t) & = & \rho_{A}(x(t))\Phi\big(x(t),\xi(t)\big),\\
\phi_{*}(\dot{x}(t)) & = & \rho_{B}(\phi(x(t))\xi(t),
\end{eqnarray*}
where $\rho_{A}$ and $\rho_{B}$ are, respectively, the anchor maps
of $A$ and $B$. The immersion $\iota$ from $\Sigma(\gamma_{(\phi,\Phi)})$
into $\Sigma(A)\times\Sigma(B)$ is the groupoid morphism given explicitly
by 
\begin{equation}
\iota:[(g,h)]\rightarrow([g],[h]),\label{eq: immersion}
\end{equation}
where $[(g,h)]$ is the class of admissible paths up to homotopy in
$\gamma_{(\phi,\Phi)}$, while $([g],[h])$ is the corresponding pair
of admissible paths up to homotopy in $A$ and $B$, respectively.

\subsection{Homotopy lifting property} \label{sec:lifting}

In this section, we prove some lifting properties for admissible paths
and homotopies via complete comorphisms. Let us start by stating two
simple facts concerning complete Lie algebroid sections and comorphisms:
\begin{itemize}
\item Let $s_{t}^{A}$ and $s_{t}^{B}$ be (time-dependent) complete sections
of Lie algebroids $A\rightarrow X$ and $B\rightarrow Y$, respectively.
Then, $\tilde{s}_{t}=s_{t}^{A}\times s_{t}^{B}$ is a complete section
of the Lie algebroid product $A\times B\rightarrow X\times Y$.
\item For $i=1,2,$ let $(\phi_{i},\Phi_{i})$ be complete comorphisms from
$A_{i}\rightarrow X_{i}$ to $B_{i}\rightarrow Y_{i}$. Then $(\phi_{1}\times\phi_{2},\Phi_{1}\times\Phi_{2})$
is a complete comorphism from $A_{1}\times A_{2}\rightarrow X_{1}\times X_{2}$
to $B_{1}\times B_{2}\rightarrow Y_{1}\times Y_{2}$. 
\end{itemize}

\begin{lem}
\label{lem: lift of time-dep sections}Let $(\phi,\Phi)$ be a complete
Lie algebroid comorphism from $A\rightarrow X$ to $B\rightarrow Y$
and let $s_{t}:Y\rightarrow B$ be a complete (time-dependent) section
of $B$. Then
\[
(\Phi^{\dagger}s_{t})(x)=\Phi(x,s_{t}(\phi(x)))
\]
is a complete (time-dependent) section of $A$. 
\end{lem}

\begin{proof}
To remove the time-dependency, we can lift $s_{t}$ to the Lie Algebroid
$B\times T\R\rightarrow Y\times\R$ by considering the section
\[
\tilde{s}(y,t)=s_{t}(y)+\partial_{t},
\]
which remains complete but which is now time-independent. The product
$\tilde{\Phi}=(\phi\times\id_{\R},\Phi\times\id_{T\R})$ is a complete
comorphism from $A\times T\R$ to $B\times T\R$, since both factors
are complete comorphisms. Thus, the lift
\[
(\tilde{\Phi}^{\dagger}\tilde{s})(x,t)=(\Phi^{\dagger}s_{t})(x)+\partial_{t}
\]
is a complete section of $A\times T\R$, and the induced flow on $X\times\R$
\[
\tilde{\Psi}_{t}(x,t)=(\Psi_{t}(x),t)
\]
exists for all $x\in X$ and all times $t\in\R$. Since $\Psi_{t}$
is the flow generated by the section $\Phi^{\dagger}s_{t}$, this
implies that $\Phi^{\dagger}s_{t}$ is complete. 
\end{proof}

\begin{prop}
(Path lifting.) \label{prop: hom. lifting prop. I}Let $(\phi,\Phi)$
be a complete comorphism from the Lie algebroid $A\rightarrow X$
to the Lie algebroid $B\rightarrow Y$, and let \textup{$g(t)=(y(t),\xi(t))$
be an admissible path in }$\mathcal{P}(B)$. Then, through any point
$x\in\phi^{-1}(y(0))$, there exists a smooth curve $x(t)$ starting
at $x$, which projects onto $y(t)$ via $\phi$, and such that
\[
\tilde{g}(t):=\big(x(t),\;\Phi\big(x(t),\xi(t)\big)\big)
\]
is an admissible path in $\mathcal{P}(A)$. 
\end{prop}

\begin{proof}
It is enough to show that we can lift the admissible path $g$ piecewise
in coordinate patches. In a local chart, we can regard the base component
$y(t)$ of a admissible path $g(t)=(y(t),\xi(t))$ as being the integral
curve of a time-dependent vector field; namely, 
\begin{equation}
X_{t}(y)=\rho_{B}(y)(\chi(y)\tilde{\xi}(t)),\label{eq:aaa}
\end{equation}
where we consider $s_{t}(y):=\chi(y)\tilde{\xi}(t)$ to be a local
(time-dependent) section of $B$. In \eqref{eq:aaa}, $\chi$ is a
cutoff function that vanishes outside a compact containing the image
of the curve $y(t)$ and that is equal to $1$ on a smaller compact
containing it, and $\tilde{\xi}$ is a smooth extension of $\xi$
to $\R$ that coincides with $\xi$ on $[0,1]$. 

The idea is to pullback \eqref{eq:aaa} to a vector field on $X$
and to obtain the lift of our admissible path as an integral curve
of this new vector field. 

Because of the cutoff function, $X_{t}$ is compactly supported and
thus complete. By Lemma \ref{lem: lift of time-dep sections}, we
obtain that $\Phi^{\dagger}s_{t}$ is complete. Thus the integral
curve $x(t)$ of $\rho_{A}\Phi^{\dagger}s_{t}$ starting at the point
\[
x(0)\in\phi^{-1}(y(0)),
\]
exists for all $t$, and, in particular, for all $t\in[0,1]$. On
this interval, we have that
\begin{eqnarray*}
\dot{x}(t) & = & \rho_{A}(x(t))\Phi\Big(x(t),\,\xi(t)\Big),
\end{eqnarray*}
which shows that 
\[
\Big(x(t),\,\Phi\big(x(t),\,\xi(t)\big)\Big)
\]
is an admissible path that lifts the one we started with. 
\end{proof}
Now we can apply Proposition \ref{prop: hom. lifting prop. I} to
homotopies
\[
g(t,s)=(y(t,s),\xi(t,s))
\]
between admissible paths in $\mathcal{P}(B)$. By definition of homotopy,
the path $g_{s}:t\mapsto g(t,s)$ is an admissible path in $\mathcal{P}(B)$
for each fixed value $s\in[0,1]$ of the homotopy parameter. Then,
given a complete comorphism $(\phi,\Phi)$ from $A$ to $B$ and a
starting point $x\in\phi^{-1}(y(0,0))$, Proposition \ref{prop: hom. lifting prop. I}
gives us a family of admissible paths,
\begin{equation}
\tilde{g}_{s}(t)=\Big(x(t,s),\,\Phi\big(x(t,s),\,\xi(t,s)\big)\Big),\label{eq:hom. lift}
\end{equation}
in $\mathcal{P}(A)$ indexed by $s\in[0,1]$, and such that $\phi(x(t,s))=y(t,s)$
for all $t$ and $s$. 
\begin{prop}
\label{prop: hom. lifting. prop. II}(Homotopy lifting.) The family
$\tilde{g}(t,s):=\tilde{g}_{s}(t)$ as above is a homotopy between
admissible paths in $\mathcal{P}(A)$. \end{prop}
\begin{proof}
Consider the Lie algebroid product $\tilde{A}:=A\times TI\times TJ$
(resp. $\tilde{B}:=B\times TI\times TJ$). We denote by $t$ the variable
in $I=[0,1]$ and by $s$ the variable in $J=[0,1]$. We introduce
the following local section of $\tilde{B}$:
\begin{eqnarray*}
s_{\xi}(y,t,s) & = & \xi(t,s)+\partial_{t},\\
s_{\beta}(y,t,s) & = & \beta(t,s)+\partial_{s},
\end{eqnarray*}
where $\xi(t,s)$ is the fiber component of the homotopy $g(t,s)$
and where $\beta(t,s)$ is the local expression of the associated
section $\beta_{t}$, restricted to $y(t,s)$. A straightforward computation
tells us that $[s_{\xi},s_{\beta}]_{\tilde{A}}=0$ since $g$ is a
homotopy. Moreover, for each fixed $s$ the curve $t\mapsto(y(t,s),t,s)$
is an integral curve of $\tilde{\rho}_{B}s_{\xi}$, while, for each
fixed $t$, the map $s\mapsto(y(t,s),t,s)$ is an integral curve of
$\tilde{\rho}_{B}s_{\beta}$, where $y(t,s)$ is the base component
of the homotopy $g(t,s)$. 

We can now lift the local sections $s_{\xi}$ and $s_{\beta}$ of
$\tilde{B}$ to local sections $\tilde{\Phi}^{\dagger}s_{\xi}$ and
$\tilde{\Phi}^{\dagger}s_{\beta}$ of $\tilde{A}$ via the comorphism
$(\tilde{\phi},\tilde{\Phi})$ from $\tilde{A}$ to $\tilde{B}$ defined
by
\[
\tilde{\phi}=\phi\times\id_{I}\times\id_{J},\quad\tilde{\Phi}=\Phi\times\id_{TI}\times\id_{TJ}.
\]
Because $[s_{\xi},s_{\beta}]_{\tilde{A}}=0$ and because $\tilde{\Phi}^{\dagger}$
and $\tilde{\rho}_{A}$ are Lie morphisms, we obtain that 
\begin{equation}
[\tilde{\Phi}^{\dagger}s_{\xi},\tilde{\Phi}^{\dagger}s_{\beta}]=0\quad\textrm{and}\quad[\tilde{\rho}_{A}\tilde{\Phi}^{\dagger}s_{\xi},\tilde{\rho}_{A}\tilde{\Phi}^{\dagger}s_{\beta}]=0.\label{eq: vanishing}
\end{equation}
Now consider the family of curves
\[
\gamma(t,s):=\Big(x(t,s),\, t,\, s\Big),
\]
where $x(t,s)$ is the base component of the lift $\tilde{g}(t,s)$
in \eqref{eq:hom. lift}. A straightforward computation shows that
the curve $\gamma(\cdot,s):t\mapsto\gamma(t,s)$ is an integral curve
of $\tilde{\rho}_{A}\tilde{\Phi}^{\dagger}s_{\xi}$ for each $s\in[0,1]$
(namely, we obtained these curves as lifts of admissible paths in
$B$ for each $s$, and thus they are admissible paths in $A$ for
each $s$). Similarly, a direct computation gives that the curve $\gamma(0,\cdot):s\mapsto\gamma(0,s)$
is an integral curve of $\tilde{\rho}_{A}\tilde{\Phi}^{\dagger}s_{\beta}$
(this relies mostly on the fact that $\beta_{0}=0$ and that $x(0,s)=x$
is constant). Since the vector fields $\tilde{\rho}_{A}\tilde{\Phi}^{\dagger}s_{\xi}$
and $\tilde{\rho}_{A}\tilde{\Phi}^{\dagger}s_{\beta}$ commute, the
family of integral curves of $\tilde{\rho}_{A}\tilde{\Phi}^{\dagger}s_{\beta}$
starting at $\gamma(t,0)$ for $t\in[0,1]$ coincide with the family
$\gamma(t,\cdot):s\mapsto\gamma(t,s)$, which implies, in particular,
that $x(t,s)$ satisfies equation \eqref{eq: x-comp}. Now the vanishing
in \eqref{eq: vanishing} implies the second homotopy equation \eqref{eq: p-comp}
by direct computation. 
\end{proof}

\begin{cor}
\label{cor: homotopy}Let $A\rightarrow X$ and $B\rightarrow Y$
be two integrable Lie algebroids, and let $(\phi,\Phi)$ be a complete
comorphism from $A$ to $B$. For all $g\in\mathcal{P}(B)$ with source
$r_{B}(g)$ in the image of $\phi$, we have that $(\tilde{g},g)\in\mathcal{P}(\gamma_{(\phi,\Phi)})$,
where $\tilde{g}$ is a lift of $g$ through $x\in\phi^{-1}(s_{B}(g))$.
Moreover, if $h\sim g$, then $(\tilde{h},h)\sim(\tilde{g},g)$. 
\end{cor}

\subsection{Analogy with Serre fibrations}

There is a certain similarity between complete comorphisms and ``Serre
fibrations'' in topology. Namely, a Serre fibration is a continuous
map $\phi:X\rightarrow Y$ between topological spaces (more precisely
CW-complexes) such that for all $n\geq0$, $f:I^{n}\rightarrow X$
and $g:I^{n}\times I\rightarrow Y$ satisfying $\phi\circ f=g\circ i_{n}$,
where $i_{n}:I^{n}\rightarrow I^{n}\times I$ is the inclusion given
by $i_{n}(\vec{t})=(\vec{t},0)$, there exists $\tilde{g}$ that makes
the following diagram commute:

\begin{diagram} 
  I^n              & \rTo^{f}            & X              \\
  \dTo^{i_n}       & \ruTo^{\tilde{g}}   & \dTo_{\phi}    \\
  I^n\times I      & \rTo_{g}            & Y
\end{diagram}
When $n=0$, $I^{0}=\{\star\}$, and we obtain the path lifting property
for $\phi$; when $n=1$, we obtain the homotopy lifting property
for homotopies between paths in $Y$. 

The analogy comes from the following facts:

\begin{itemize}
\item An admissible path in the algebroid $A\rightarrow X$ is the same
thing as a Lie algebroid morphism from $TI$ to $A$;
\item A homotopy between admissible paths is a Lie algebroid morphism from
$TI\times TI$ to $A$;
\item The tangent map to the inclusion $i_{n}$ is a Lie algebroid morphism
from $TI^{n}$ to $TI^{n}\times TI$.
\end{itemize}

With this in mind, Propositions \ref{prop: hom. lifting prop. I}
and \ref{prop: hom. lifting. prop. II} can be summarized diagrammatically
(for $n=0,1$) as follows:
\begin{diagram} 
  TI^n               & \rTo^{f}            & A              \\
  \dTo^{Ti_n}        & \ruTo^{\tilde{g}}   & \dTo_{\Phi}    \\
  TI^n\times TI      & \rTo_{g}            & B
\end{diagram}
where $\Phi$ is a complete algebroid comorphism from Lie algebroids
$A\rightarrow X$ to $B\rightarrow Y$. For $n=0$, $g$ is an admissible
path, and, for $n=1$, $g$ is a homotopy between admissible paths. 

The problem with the diagram above is that its arrows do not belong
to the same category, since it involves morphisms and comorphisms
of Lie algebroids. Going beyond a mere analogy would require a category
whose objects are the Lie algebroids and whose morphisms comprise
both morphisms and comorphisms of Lie algebroids.

\section{The integration functor \label{sec:Dazord's-Theorem}}

In \cite{da:groupoide}, Dazord announced, without proof (\cite[Thm. 4.1]{da:groupoide}),
that a complete comorphism between integrable Lie algebroids always
integrates to a unique comorphism between the integrating Lie groupoids.
We will prove this result here using the path construction, which,
together with Proposition \ref{prop: Lie}, yields an improvement
of Dazord's Theorem: namely, that a Lie algebroid comorphism is integrable
\textit{if and only if} it is complete. 

Actually, we will show that the classical integration functor for
Lie algebras generalizes to integrable Lie algebroids and complete
comorphisms:

\begin{thm}
\label{thm: Dazord}The path construction $\Sigma$ is a functor from
the category of integrable Lie algebroids and complete comorphisms
to the category of source $1$-connected Lie groupoids and comorphisms.
It is an inverse to the Lie functor $\Lie$, and, thus, implements
an equivalence between these two categories. 
\end{thm}

As corollary of Theorem \ref{thm: Dazord}, Proposition \ref{prop: Lie},
and Corollary \ref{prop:faithfulness} (for the uniqueness part),
we obtain Dazord's statement:

\begin{cor}
(Dazord \cite{da:groupoide}). Let $A\rightarrow X$ and $B\rightarrow Y$
be two integrable Lie algebroids with source $1$-connected integrating
Lie groupoids $G$ and $H$. Then a Lie algebroid comorphism from
$A$ to $B$ integrates to a (unique) Lie groupoid comorphism from
$G$ to $H$ if and only if it is complete. 
\end{cor}

The rest of this section is devoted to the proof of Theorem \ref{thm: Dazord}.
In Paragraph \ref{sub:Embeddability}, we show that $\Sigma$ takes
a complete Lie algebroid comorphism to a Lie groupoid comorphism.
In Paragraph \ref{sub:Functoriality}, we show that $\Sigma$ is functorial,
and in Paragraph \ref{sub:Embeddability} we show that it is a homotopy
inverse to the Lie functor. 

\begin{rem}
One way to prove that a complete Lie algebroid comorphism integrates to 
a Lie groupoid comorphism would be to adapt the
corresponding proof for complete Poisson maps of Caseiro and Fernandes (Prop. 4.8 an
Prop. 4.9 in \cite{CF}) to comorphisms and to show that the resulting embedded subgroupoid
is the graph of a Lie groupoid comorphism. We give in Paragraph \ref{sub:Embeddability}
a different proof, which, however, relies also on the same kind of lifting properties as in \cite{CF}.
\end{rem}

\subsection{Embeddability\label{sub:Embeddability}}

Recall that any source $1$-connected Lie groupoid integrating a Lie
algebroid $A\rightarrow X$ is isomorphic to the Lie groupoid $\Sigma(A)$
obtained by the path construction. Therefore, in order to prove the
first part of the theorem, it is enough to show that the immersion
\begin{eqnarray*}
\iota:\Sigma(\gamma_{(\phi,\Phi)}) & \longrightarrow & \Sigma(A)\times\Sigma(B),\\
{}[(g,h)] & \mapsto & ([g],[h]),
\end{eqnarray*}
defined in \eqref{eq: immersion} (i.e., the hypercomorphism integrating
the comorphism $(\phi,\Phi)$ from $A$ to $B$) is a closed embedding
whose image is the graph of a comorphism $(\phi,\Psi)$ from $\Sigma(A)$
to $\Sigma(B)$, when $(\phi,\Phi)$ is complete. 

We start with a series of lemmas, in which $A$ and $B$ will always
be integrable, and $(\phi,\Phi)$ will always be a complete comorphism
of Lie algebroids from $A$ to $B$. 

\begin{lem}
\label{lem: injective immersion}The immersion $\iota:\Sigma(\gamma_{(\phi,\Phi)})\rightarrow\Sigma(A)\times\Sigma(B)$
defined in \eqref{eq: immersion} is injective. Moreover, we have
that
\begin{equation}
\Sigma_{\iota}(\gamma_{(\phi,\Phi)})\cap(\Sigma(A)\times Y)=\graph\phi,\label{eq: transversality}
\end{equation}
where $\Sigma_{\iota}(\gamma_{(\phi,\Phi)}):=\iota(\Sigma(\gamma_{(\phi,\Phi)}))$.
\end{lem}

\begin{proof}
Since $\iota$ is also a groupoid morphism, we only need, in order
to show injectivity, to check that if $\iota([\gamma])$ is a groupoid
unit, then so is $[\gamma]$. For this, we start with a representative
of $[\gamma]$, say,
\[
\gamma:t\mapsto\Big(\big(x(t),\,\Phi(x(t),\xi(t))\big),\,\big(\phi(x(t)),\,\xi(t)\big)\Big).
\]
Let us denote by $\gamma_{A}$ (reps. $\gamma_{B}$) the component
of $\gamma$ in $\mathcal{P}(A)$ (resp. $\mathcal{P}(B)$). Because
$\iota$ sends $[\gamma]$ to a unit, there is a homotopy
\[
\gamma_{B}\quad\overset{s=0}{\longleftarrow}\quad g(t,s):=(y(t,s),\,\xi(t,s))\quad\overset{s=1}{\longrightarrow}\quad(y,0)
\]
between $\gamma_{B}$ and a constant path. Using Proposition \ref{prop: hom. lifting. prop. II},
we can lift $g$ to a homotopy of the form
\[
\tilde{g}(t,s):=\big(x(t,s),\,\Phi(x(t),\xi(t,s))\big)\quad\overset{s=0}{\longrightarrow}\quad\gamma_{A}
\]
among the admissible paths in $\mathcal{P}(A)$, such that $\phi(x(t,s))=y(t,s)$.
This homotopy relates $\gamma_{A}$ to a constant path $(0,x)$, for
we have that $\xi(t,1)=0$ for all $t\in[0,1]$. Putting the homotopies
for the two components of $\iota([\gamma])$ together, we obtain a
homotopy
\[
\Big(\big(x(t,s),\,\Phi(x(t),\xi(t,s))\big),\,\big(\phi(x(t,s)),\,\xi(t,s)\big)\Big)
\]
in $\mathcal{P}(\gamma_{(\phi,\Phi)})$ (see Corollary \ref{cor: homotopy})
between $[\gamma]$ and a unit in $\Sigma(\gamma_{(\phi,\Phi)})\rightrightarrows\graph\phi$. 

The very same argument also proves the second part of the statement,
since we only needed $[\gamma_{B}]$ to be a unit in $\Sigma(B)\rightrightarrows Y$
in order to prove that $[(\gamma_{A},\gamma_{B})]$ is a unit in $\Sigma(\gamma_{(\phi,\Phi)})\rightrightarrows\graph\phi$. \end{proof}

\begin{lem}
\label{lem: graphs}The intersection of $\Sigma_{\iota}(\gamma_{(\phi,\Phi)})$
with $r_{A}^{-1}(x)\times r_{B}^{-1}(\phi(x))$ is the graph of a
map $\psi_{x}:r_{B}^{-1}(\phi(x))\rightarrow r_{A}^{-1}(x)$ for all
$x\in A$. 
\end{lem}

\begin{proof}
Suppose that we have two elements
\[
([g],[h]),\,([g'],[h])\in\Sigma_{\iota}(\gamma_{(\phi,\Phi)})\cap\big(r_{A}^{-1}(x)\times r_{B}^{-1}(\phi(x))\big).
\]
Consider the corresponding preimages $[(g,h)]$ and $[(g',h)]$ in
$\Sigma(\gamma_{(\phi,\Phi)})$ (which are unique, since $\iota$
is injective). Because $\Sigma(\gamma_{(\phi,\Phi)})$ is a groupoid,
the inverse $[(g^{-1},h^{-1})]$ of $[(g,h)]$ is in $\Sigma(\gamma_{(\phi,\Phi)})$,
and so is the product $[(g'g^{-1},hh^{-1})]$. Using Lemma \ref{lem: injective immersion}
with
\[
\iota([g'g^{-1},hh^{-1}])=([g'g^{-1}],[hh^{-1}])\in\Sigma_{\iota}(\gamma_{(\phi,\Phi)})\cap\big(\Sigma(A)\times Y\big),
\]
we obtain that $[g'g^{-1}]$ is a unit in $\Sigma(A)$, and thus $[g']=[g]=:\psi_{x}([h])$. 
\end{proof}

\begin{lem}
\label{lem: uniq of liftings}Let $g\in\mathcal{P}(B)$ be an admissible
path, and suppose that we have two lifts $\tilde{g}_{1},\tilde{g}_{2}\in\mathcal{P}(A)$
of $g$, as in Proposition \ref{prop: hom. lifting prop. I}, such
that $r_{A}([\tilde{g}_{1}])=r_{A}([\tilde{g}_{2}])$ (i.e., they
are lifts of $g$ through the same point $x\in X$, which is the source
of both paths $\tilde{g}_{1}$ and $\tilde{g}_{2}$). Then $[\tilde{g}_{1}]=[\tilde{g}_{2}]$. 
\end{lem}

\begin{proof}
Suppose that $[\tilde{g}_{1}]\neq[\tilde{g}_{2}]$. Then both $[(\tilde{g}_{1},g)]$
and $[(\tilde{g}_{2},g)]$ are admissible paths in $\mathcal{P}(\gamma_{(\phi,\Phi)})$
by Corollary \ref{cor: homotopy}. Thus, we obtain two different elements
$([\tilde{g}_{1}],[g])$ and $([\tilde{g}_{2}],[g])$ in $\Sigma_{\iota}(\gamma_{(\phi,\Phi)})$,
which contradicts  Lemma \ref{lem: uniq of liftings}. 
\end{proof}

Since $r_{B}$ is a surjective submersion, the fiber product $X\times_{Y}\Sigma(B)$
is a closed submanifold of $\Sigma(A)\times\Sigma(B)$. Now Lemma
\ref{lem: graphs} shows that $\Sigma_{\iota}(\gamma_{(\phi,\Phi)})$
is the image of a section $\sigma$ of the fibration
\[
r_{A}\times\id:\Sigma(A)\times\Sigma(B)\rightarrow X\times\Sigma(B)
\]
over the closed submanifold $X\times_{Y}\Sigma(B)$, and where $\sigma(x,[g]):=(x,\psi_{x}([g]))$,
and thus $\Sigma_{\iota}(\gamma_{(\phi,\Phi)})$ is a closed submanifold.
We have thus shown that $\Sigma_{\iota}(\gamma_{(\phi,\Phi)})$ is
the graph of a comorphism $(\phi,\Psi)$ from $\Sigma(A)$ to $\Sigma(B)$,
where
\[
\Psi(x,[g])=\psi_{x}([g])=[\tilde{g}]
\]
with $[\tilde{g}]$ being the unique homotopy class of lifts of $g$
through $x$.

\subsection{Functoriality\label{sub:Functoriality}}

As we explained in Section \ref{sub: path construction}, the path
construction $\Sigma$ associates a source $1$-connected Lie groupoid
$\Sigma(A)$ with an \textit{integrable} Lie algebroid $A$. In the
previous paragraph, we showed that $\Sigma$ associates the comorphism
$\Sigma(\gamma_{(\phi,\Phi)})$ from $\Sigma(A)$ to $\Sigma(B)$
with a \textit{complete} comorphism from $A$ to $B$. 

We want to show that $\Sigma$ is a functor from the category of integrable
Lie algebroids and complete comorphisms to the category of source
$1$-connected Lie groupoids and comorphisms. For this, we need to
show that
\[
\Sigma(R_{2})\circ\Sigma(R_{1})=\Sigma(R_{3}),
\]
where $R_{1}$ is the graph of a comorphism $(\phi_{1},\Phi_{1})$
from $A$ to $B$, $R_{2}$ is the graph of a comorphism $(\phi_{2},\Phi_{2})$
from $B$ to $C$, and $R_{3}$ is the graph of the composition of
$(\phi_{1},\Phi_{1})$ with $(\phi_{2},\Phi_{2})$. The bases of the
integrable Lie algebroids $A$, $B$, and $C$ are, respectively,
$X$, $Y$, and $Z$. 

Recall that the composition of comorphisms between Lie algebroids,
Lie groupoids, or, more generally, between fibrations $r_{A}:A\rightarrow X$,
$r_{B}:B\rightarrow Y$ and $r_{C}:C\rightarrow Z$ is given by
\begin{eqnarray*}
(\phi_{2},\Psi_{2})\circ(\phi_{1},\Psi_{1}) & = & (\phi_{2}\circ\phi_{1},\Psi_{1}\star\Psi_{2}),\\
(\Psi_{1}\star\Psi_{2})(x,c) & = & \Psi_{1}(x,\Psi_{2}(\phi_{1}(x),c)),
\end{eqnarray*}
for $x\in X$ and $c\in r_{C}^{-1}(\phi_{2}\circ\phi_{1}(x))$. This
composition translates in terms of the comorphism graphs $R_{1}$
and $R_{2}$ into the composition of the underlying binary relations:
i.e., the graph of the comorphism composition is the relation $R_{2}\circ R_{1}$
in $A\times C$ obtained by projecting the image of
\[
(R_{1}\times R_{2})\cap(A\times\Delta_{B}\times C),
\]
where $\Delta_{B}$ is the diagonal in $B\times B$, to $A\times C$.
The fact that these relations come from comorphism graphs guarantees
that the result of the composition is a closed submanifold of $A\times C$. 

\begin{rem}
There are three ways of looking at a Lie algebroid comorphism from
$A$ to $B$ $ $: (1) as the pair $(\phi,\Phi)$; (2) as the underlying
relation $\gamma_{(\phi,\psi)}\subset A\times B$; (3) as the corresponding
Lie algebroid $\gamma_{(\phi,\Phi)}\rightarrow\graph\phi$. Similarly,
there are three ways of looking at a Lie groupoid comorphism from
$G$ to $H$: (1) as the pair $(\phi,\Phi)$; (2) as the underlying
relation $R_{(\phi,\Phi)}\subset G\times H$; (3) as the Lie groupoid
$R_{(\phi,\Phi)}\rightrightarrows\graph\phi$. 
\end{rem}

\begin{lem}
$\Sigma(R_{2})\circ\Sigma(R_{1})$ contains $\Sigma(R_{2}\circ R_{1})$. 
\end{lem}

\begin{proof}
Given $[\gamma]\in\Sigma(R_{2}\circ R_{1})$, we will exhibit an element
$[\gamma_{1}]\times[\gamma_{2}]\in\Sigma(R_{1})\times\Sigma(R_{2})$,
whose image by the projection
\begin{equation}
\Sigma(A)\times\Delta_{\Sigma(B)}\times\Sigma(C)\rightarrow\Sigma(A)\times\Sigma(C)\label{eq:reduction}
\end{equation}
is precisely $[\gamma]$. Namely, a representative of $[\gamma]\in\Sigma(R_{2}\circ R_{1})$
is of the form
\begin{eqnarray*}
\gamma:t & \mapsto & \Big(x(t),\,(\Phi_{1}\star\Phi_{2})(x(t),\xi(t)),\,(\phi_{2}\circ\phi_{1})(x(t)),\,\xi(t)\Big)
\end{eqnarray*}
for some path $t\mapsto(x(t),\xi(t))$. We set
\begin{eqnarray*}
y(t) & := & \phi_{1}(x(t)),\\
\eta(t) & := & \Phi_{2}(y(t),\xi(t)).
\end{eqnarray*}
This gives us two representatives of paths, 
\begin{gather*}
\gamma_{1}=(g_{1},h_{1})\textrm{ in }\Sigma(R_{1}),\\
\gamma_{2}=(g_{2},h_{2})\textrm{ in }\Sigma(R_{2}),
\end{gather*}
respectively given by 
\begin{eqnarray*}
\gamma_{1}:t & \mapsto & \Big(x(t),\Phi_{1}(x(t),\eta(t)),\,\phi_{1}(x(t)),\,\eta(t)\Big),\\
\gamma_{2}:t & \mapsto & \Big(y(t),\Phi_{2}(y(t),\xi(t)),\,\phi_{2}(y(t)),\,\xi(t)\Big).
\end{eqnarray*}
Since $h_{1}=g_{2}$ by definition, we obtain that
\[
[\gamma_{1}]\times[\gamma_{2}]\in\Sigma(A)\times\Delta_{\Sigma(B)}\times\Sigma(C),
\]
and thus $[\gamma_{1}]\times[\gamma_{2}]$ projects via \eqref{eq:reduction}
on the equivalence class of the path
\begin{eqnarray*}
t & \mapsto & \Big(x(t),\,\Phi_{1}(x(t),\eta(t)),\,\phi_{2}(y(t)),\,\xi(t)\Big),
\end{eqnarray*}
which we recognize to be precisely $\gamma$ since 
\[
(\Phi_{1}\star\Phi_{2})(x(t),\xi(t))=\Phi_{1}(x(t),\Phi_{2}(\phi_{1}(x(t)),\xi(t))).
\]
\end{proof}

\begin{lem}
$\Sigma(R_{2}\circ R_{1})$ contains $\Sigma(R_{2})\circ\Sigma(R_{1})$. 
\end{lem}

\begin{proof}
For this, consider
\begin{eqnarray*}
[\gamma] & = & [(\gamma_{A},\gamma_{B})]\in\Sigma(R_{1}),\\
{}[\delta] & = & [(\delta_{B},\delta_{C})]\in\Sigma(R_{2}),
\end{eqnarray*}
such that 
\[
([\gamma],[\delta])\in\Sigma(A)\times\Delta_{\Sigma(B)}\times\Sigma(C).
\]
This means that $[\gamma_{B}]$ and $[\delta_{B}]$ are in the same
homotopy class of paths in $\mathcal{P}(B)$. Thus, there is a homotopy
\[
\gamma_{B}\overset{\nu_{B}}{\rightsquigarrow}\delta_{B},\quad\nu_{B}(t,s):=\big(y(t,s),\eta(t,s)\big),
\]
from $\gamma_{B}$ to $\delta_{B}$. The homotopy lifting property
(Proposition \ref{prop: hom. lifting. prop. II}) tells us that we
can lift $\nu_{B}$ to a homotopy $\mu_{A}$ (among the admissible
paths in $\mathcal{P}(A)$) of the form
\[
\gamma_{A}\overset{\mu_{A}}{\rightsquigarrow\delta_{A}},\quad\mu_{A}(t,s):=\big(x(t,s),\,\Phi_{1}(x(t,s),\,\eta(t,s))\big),
\]
such that
\begin{eqnarray*}
\mu_{A}(t,0) & = & \gamma_{A}(t),\\
\phi_{1}(x(t,s)) & = & y(t,s),
\end{eqnarray*}
and where we have set $\delta_{A}:=\mu_{A}(t,1)$. Now putting $\nu_{B}$
and $\mu_{A}$ together, we obtain with Corollary \ref{cor: homotopy}
the homotopy
\[
(\gamma_{A},\gamma_{B})\overset{\theta}{\rightsquigarrow}(\delta_{A},\delta_{B}),\quad\theta:=(\mu_{A},\nu_{B}),
\]
among the paths in $\mathcal{P}(R_{1})$. Thus, we can take $(\delta_{A},\delta_{B})$
as a representative of $[\gamma]\in\Sigma(R_{1})$. Finally, we see
that the representative
\[
([\gamma],[\delta]):=\Big([(\delta_{A},\delta_{B})],\,[(\delta_{B},\delta_{C})]\Big)
\]
projects, after reduction, to $[(\delta_{A},\delta_{C})]$, which
belongs to $\Sigma(R_{2}\circ R_{1})$. 
\end{proof}

\subsection{Equivalence\label{sub:Equivalence}}

We now prove that the functor $\Sigma$ is a homotopy inverse to the
Lie functor $\Lie$. 

Since, by construction, $\Lie\Sigma(A)=A$ for any integrable Lie
algebroid $A$, we have that $\Lie\circ\Sigma$ is the identity functor.
We need to show that $\Sigma\circ\Lie$ is also homotopic to the identity
functor. 

Given a source $1$-connected Lie groupoid $G$, it is well-known
that $\Sigma(\Lie(G))$ is isomorphic, as groupoid, to $G$ (see \cite{CF2003}
for instance). Moreover for each such $G$, there is a canonical groupoid
isomorphism $\alpha_{G}:\Sigma(\Lie(G))\rightarrow G$, which is constructed
as follows:

To begin with, consider the monodromy groupoid $\hat{G}$ of $G$,
whose elements are end-point-fixing homotopy classes $[\gamma]$ of
$G$-paths, that is, paths $\gamma:[0,1]\rightarrow G$ such that
$\gamma(0)$ is a groupoid unit $x$ and $\gamma(t)$ stays in the
source-fiber of $x$ (i.e. $r_{G}(\gamma(t))=x$). When $G$ is source
$1$-connected, the map
\[
{\bf ev}:\hat{G}\rightarrow G:\;[\gamma]\mapsto\gamma(1)
\]
is a groupoid isomorphism. There is also a groupoid isomorphism $D$
from the monodromy groupoid $\hat{G}$ to $\Sigma(\Lie(G))$. Namely,
given a $G$-path $\gamma$, 
\[
D(\gamma)(t)=T_{l_{G}(\gamma(t))}L_{\gamma(t)^{-1}}\dot{\gamma}(t),
\]
where $L_{g}(h)=gh$ is the groupoid left-translation, is an admissible
path in $\Lie(G)$. It turns out that $D$ preserves the homotopy
classes of $G$-paths and $A$-paths and defines a groupoid isomorphism
with inverse $D^{-1}:\Sigma(\Lie(G))\rightarrow\hat{G}$. We refer
the reader to \cite{CF2003} for details and proofs. 

For each source $1$-connected groupoid, we can now define the groupoid
isomorphism 
\[
\alpha_{G}:={\bf ev}\circ D^{-1}:\;\Sigma(\Lie(G))\rightarrow G.
\]
Since isomorphisms in $\Gpdplus$ and $\Gpdminus$ coincide, $\alpha_{G}$
is also a comorphism. 

Let us show now that the $\alpha_{G}$'s are the components of a natural
isomorphism between the identity functor and $\Sigma\circ\Lie$; in
other words, that, for any Lie groupoid comorphism $(\phi,\Phi)$
from $G$ to $H$, the following diagram commutes:

\begin{diagram} 
  \Sigma\circ\Lie(G)             & \rTo^{\alpha_G}            & G              \\
  \dTo^{\Sigma\circ\Lie(\phi,\Phi)}        &                     & \dTo_{(\phi,\Phi)}    \\
  \Sigma\circ\Lie(H)     & \rTo_{\alpha_H}            & H
\end{diagram}

This follows from the following lemmas.

\begin{lem}
\label{lem: comutation in bin}The diagram above commutes if and only
if 
\begin{equation}
(\alpha_{G}\times\alpha_{H})(\gamma_{\Sigma\circ\Lie(\phi,\Phi)})=\gamma_{(\phi,\Phi)},\label{eq:natural isom}
\end{equation}
where $\gamma_{\Sigma\circ\Lie(\phi,\Phi)}$ and $\gamma_{(\phi,\Phi)}$
are the graphs of the corresponding comorphisms.
\end{lem}

\begin{proof}
This comes from the fact that composition of comorphisms is the same
as the composition of their underlying graphs as binary relations
(i.e. $\gamma_{(\phi_{1},\Phi_{1})}\circ\gamma_{(\phi_{2},\Phi_{2})}=\gamma_{(\phi_{1},\Phi_{1})\circ(\phi_{2},\Phi_{2})}$).
Namely, one checks that for sets $G,H,G',H'$ , bijections $\alpha_{G}:G'\rightarrow G$
and $\alpha_{H}:H'\rightarrow H$ , and binary relations $R\subset G\times H$
and $R'\subset G'\times H'$ , we have $R\circ\graph\alpha_{G}=\graph\alpha_{H}\circ R'$
(where the composition is the composition of binary relations) if
and only if $(\alpha_{G}\times\alpha_{H})(R')=R$. 
\end{proof}

\begin{lem}
\label{lem:connectedness}Let $G\rightrightarrows X$ and $H\rightrightarrows Y$
be Lie groupoids, with $H$ source $1$-connected. Then the graph
$\gamma_{(\phi,\Phi)}$ of a comorphism $(\phi,\Phi)$ from $G$ to
$H$ (seen as a subgroupoid of $G\times H$) is source $1$-connected.
\end{lem}

\begin{proof}
We need to show that the source fiber $s^{-1}(x,\phi(x))$ in $\gamma_{(\phi,\Phi)}$
is $1$-connected for all $x\in X$. Since $\gamma_{(\phi,\Phi)}$
is a subgroupoid of $G\times H$, we have that
\[
s^{-1}(x,\phi(x))=(s_{G}\times s_{H})^{-1}(x,\phi(x))\cap\gamma_{(\phi,\Phi)}=\graph\Phi_{x}.
\]
Because the domain $s_{H}^{-1}(x)$ of $\Phi_{x}$ is $1$-connected
by assumption, so is its graph. \end{proof}

\begin{lem}
\label{lem: alphas}We have that $\alpha_{G\times H}=\alpha_{G}\times\alpha_{H},$
for $G$ and $H$ source $1$-connected Lie groupoids. Moreover, the
restriction of $\alpha_{G}$ to a source $1$-connected Lie subgroupoid
$H$ coincides with $\alpha_{H}$; in other words, the following diagram
commutes:\begin{diagram} 
  \Sigma\circ\Lie(H)             & \rTo^{k}            & \Sigma\circ\Lie(G)              \\
  \dTo^{\alpha_H}        &                     & \dTo_{\alpha_G}    \\
  H     & \rTo_{i}            & G
\end{diagram}where $k$ sends a homotopy class $[\gamma]_{H}$ of admissible paths
in $\Lie(H)$ to the homotopy class of path $[\gamma]_{G}$ in $\Lie(G)$,
and $i$ is the inclusion of $G$ in $H$. 
\end{lem}

\begin{proof}
The first statement follows from the facts that
\begin{eqnarray*}
\Sigma\circ\Lie(G\times H) & = & \Sigma\circ\Lie(G)\times\Sigma\circ\Lie(H),\\
\widehat{G\times H} & = & \widehat{G}\times\widehat{H},\\
{\bf ev}_{G\times H} & = & {\bf ev}_{G}\times{\bf ev}_{H},\\
D_{G\times H} & = & D_{G}\times D_{H}.
\end{eqnarray*}
As for the second statement, observe first that $k$ is well-defined,
since an admissible path in $\Lie(H)$ is also an admissible path
in $\Lie(G)$, and homotopic $\Lie(H)$-paths are also homotopic as
$\Lie(G)$-paths (since $\Lie(H)$-homotopies are also $\Lie(G)$-homotopies).
Moreover, an admissible $\Lie(H)$-path $\gamma$ integrates to a
$H$-path, which, considered as a $G$-path, is the same as the one
$\gamma$ integrates to when considered as a $\Lie(G)$-path. We can
then conclude by chasing in the diagram above, starting with the representative
$\gamma$. 
\end{proof}

\begin{lem}
The $\alpha$'s defined above are the components of a natural transformation.
\end{lem}

\begin{proof}
By Lemma \ref{lem: comutation in bin}, we only need to show that
the Lie groupoid comorphisms $(\alpha_{G}\times\alpha_{H})(\gamma_{\Sigma\circ\Lie(\phi,\Phi)})$
and $\gamma_{(\phi,\Phi)}$ coincide. Since $\gamma_{(\phi,\Phi)}$
is a source $1$-connected groupoid by Lemma \ref{lem:connectedness},
we have that $\alpha_{\gamma_{(\phi,\Phi)}}$ is an isomorphism from
$\gamma_{\Sigma\circ\Lie(\phi,\Phi)}$ to $\gamma_{(\phi,\Phi)}$.
From Lemma \ref{lem: alphas}, we can conclude that the restriction
of $\alpha_{G}\times\alpha_{H}=\alpha_{G\times H}$ to the subgroupoid
$\gamma_{\Sigma\circ\Lie(\phi,\Phi)}$ coincides with $\alpha_{\gamma_{(\phi,\Phi)}}$,
whose image is exactly $\gamma_{(\phi,\Phi)}$.
\end{proof}

Consequently, $\Sigma$ is a homotopy inverse to the Lie functor,
implementing thus an equivalence of categories. As corollary, we have
that

\begin{cor}
\label{prop:faithfulness}The Lie functor is faithful. In other words,
the Lie groupoid comorphism integrating a complete Lie algebroid comorphism
is unique.
\end{cor}

\section{The symplectization functor\label{sec:Application-to-poisson}}

There is an immediate application of Theorem \ref{thm: Dazord} in
Poisson geometry. Namely, this theorem implies, as we will see below,
that the integration of Poisson manifolds by symplectic groupoids
using the path construction is an actual functor from the category
of integrable Poisson manifolds and \textit{complete }Poisson maps
to the category $\SGpd$ of source $1$-connected symplectic groupoids
and \textit{symplectic comorphisms.} 

A \textbf{symplectic comorphism} from symplectic groupoids $G\rightrightarrows X$
to $H\rightrightarrows Y$ is a comorphism $(\phi,\Phi)$ whose underlying
graph $\gamma_{(\phi,\Phi)}$ is a canonical relation from $G$ to
$H$. In contrast with general canonical relations, symplectic comorphisms
always compose well (because they are comorphisms in the first place),
and thus form a category. Observe that the graph $\gamma_{(\phi,\Phi)}$
of a symplectic comorphism is a lagrangian subgroupoid of $\overline{G}\times H$

A \textbf{complete Poisson map} $\phi$ from $X$ to $Y$ is a Poisson
map with the property that the hamiltonian vector field $\xi_{\phi^{*}f}$
on $X$ with hamiltonian $\phi^{*}f$ is complete if the hamiltonian
vector field $\xi_{f}$ on $Y$ with hamiltonian $f\in C^{\infty}(Y)$
is complete. 

Fernandes in \cite{fernandes2006} studied constructions in Poisson
geometry involving the integration of Poisson manifolds seen as a
functor, which he called the ``symplectization functor.'' However,
in \cite{fernandes2006} the domain of this functor comprises all
Poisson maps and its range has for morphisms $ $from $G$ to $H$
all the lagrangian subgroupoids of $\overline{G}\times H$, instead
of only those that are graphs of symplectic comorphisms. This choice
has as a consequence that the range category is not an honest category
(the compositions are not always well-defined). Moreover, if we drop
the completeness condition, there are Poisson maps that do not integrate
to symplectic comorphisms as illustrated in the example below. Hence,
the symplectization functor is not a true functor with this choice
of domain and range.

\begin{example}
Consider the non-complete and non-integrable Lie algebroid comorphism
$(\phi,\Phi)$ from $TX$ to $TY$ of Section \ref{sec:The-example}.
Its dual $\Phi^{*}$ is a non-complete Poisson map from cotangent
bundles $T^{*}X$ to $T^{*}Y$ endowed with their canonical Poisson
structure. Let us see that $\Phi^{*}$ is also non-integrable. Since
the graph $\graph\phi^{*}$ is a coisotropic submanifold of the \textit{symplectic}
manifold $\overline{T^{*}X}\times T^{*}Y$ (with symplectic form $\Omega=-\omega+\omega$),
the (immersed) lagrangian subgroupoid integrating $\Phi^{*}$ can
be identified with the leafwise fundamental groupoid of the characteristic
foliation $\tilde{\mathcal{F}}$ of $\graph\phi^{*}$ (whose associated
distribution we denote by $\tilde{\Delta}$). For two vectors in the
tangent space to $\graph\phi^{*}$ (which we identify with vectors
$\underline{v}=v\oplus\theta_{v}$ in $TT^{*}X\simeq TX\oplus T^{*}X$),
we have that
\[
\Omega(\underline{v},\underline{w})=-\langle(\id-\Phi\circ T\phi)v,\,\theta_{w}\rangle+\langle(\id-\Phi\circ T\phi)w,\,\theta_{v}\rangle.
\]
From this last equation, we see that $\Omega(\underline{v},\underline{w})=0$
for all vectors $\underline{v}$ tangent to $\graph\phi^{*}$ iff
$\underline{w}\in\im\Phi\oplus(\ker T\phi)^{0}$, where $\im\Phi$
is the distribution of the foliation given by the flat Ehresmann connection
associated with $(\phi,\Phi)$ and $(\ker T\phi)^{0}$ is the annihilator
of vertical distribution associated with the submersion $\phi$. Hence,
\[
\tilde{\Delta}=\im\Phi\oplus(\ker T\phi)^{0},
\]
and the leafwise fundamental groupoid of $\tilde{\mathcal{F}}$ may
be parametrized by $\tilde{\Gamma}=T^{*}\Gamma$, where 
\[
\Gamma=(\reals^{+}\times S^{1})\times(\reals^{+}\times\reals)\times\complex\times(-1,1)
\]
parametrizes the leafwise fundamental groupoid of the Ehresmann connection
on $X$ as described in Section \ref{sec:The-example}. The element
\[
\tilde{\gamma}=(r,\,\theta,\, r',\,\tau,\, z,\, h,\,\xi_{r},\,\xi_{\theta},\,\xi_{r}',\,\xi_{\tau},\,\xi_{z},\,\xi_{h})
\]
of $\tilde{\Gamma}$ corresponds to the homotopy class of the leafwise
path 
\[
t\mapsto\Big(r+(r'-r)t,\,\theta+\tau t,\, e^{i\nu(h)\tau t}z,\, h,\,\xi_{r}+(\xi_{r}'-\xi_{r})t,\,\xi_{\theta}+\xi_{\tau}t,\,\xi_{z},\,\xi_{h}\Big),
\]
with $0\leq t\leq1$. We see that we obtain the same non-trivial self-intersections
for the immersion 
\[
\tilde{\Gamma}\rightarrow\overline{T^{*}X}\times T^{*}X\times\overline{T^{*}Y}\times T^{*}Y
\]
at $z=0$ for the exact same reasons as in Section \ref{sec:The-example}.
\end{example}

As with Lie algebroids, there is a Lie functor $\Lie$ from $\SGpd$
to the category of Poisson manifolds and Poisson maps. It takes a
symplectic groupoid $G\rightrightarrows X$ to the Poisson manifold
$(X,\Pi_{X})$, where $\Pi_{X}$ is the unique Poisson structure turning
$r_{G}$ into a Poisson map (and $l_{G}$ into an anti-Poisson map).
Since the graph of a symplectic comorphism $(\phi,\Phi)$ from $G\rightrightarrows X$
to $H\rightrightarrows Y$ is a lagrangian subgroupoid $\gamma_{(\phi,\Phi)}\rightrightarrows\graph\phi$,
this implies that the graph of $\phi$ is a coisotropic submanifold
(see \cite{cattaneo2004}), and, hence, that $\phi$ is a Poisson
map. The Lie functor on morphisms is thus defined as $\Lie(\phi,\Phi)=\phi$. 

The path construction can also be extended in a functorial way to
the Poisson realm. Namely, integrating a Poisson manifold $X$ is
equivalent to integrating its associated Lie algebroid $T^{*}X\rightarrow X$,
whose bracket on sections is the Koszul bracket and whose anchor map
is the map $\Pi^{\sharp}:T^{*}X\rightarrow TX$ associated with the
Poisson bivector field $\Pi\in\Gamma(\wedge^{2}TX)$. The path construction
applied to this Lie algebroid yields a symplectic groupoid $\Sigma(T^{*}X)\rightrightarrows X$
(when the Poisson manifold is integrable as seen in \cite{CF2001}),
that is, a groupoid whose total space is symplectic and whose multiplication
graph is a lagrangian subgroupoid of $\overline{\Sigma(T^{*}X)}\times\overline{\Sigma(T^{*}X)}\times\Sigma(T^{*}X)$.
(The bar on a Poisson manifold denotes the same Poisson manifold but
with opposite Poisson structure.) 

To extend the path construction to Poisson maps, we need the following
proposition, which can already be (partly) found in \cite{HM1993}:

\begin{prop}
\label{prop: complete} Let $(X,\Pi_{X})$ and $(Y,\Pi_{Y})$ be two
Poisson manifolds. A smooth map $\phi:X\rightarrow Y$ is a Poisson
map if and only if its cotangent map $\Cot\phi$ is a Lie algebroid
comorphism from $\Cot X$ to $\Cot Y$ (with the Lie algebroid structure
described above). Moreover, $\phi$ is complete if and only if $T^{*}\phi$
is.
\end{prop}

\begin{proof}
$T^{*}\phi$ is a comorphism if and only if its dual, the tangent
map $T\phi$ from $TX$ to $TY$, is a Poisson map with respect to
the Poisson structure on $TX$ and $TY$ inherited from being duals
of Lie algebroids. Thus we only need to show that $\phi$ is Poisson
if and only if $T\phi$ is. Now a smooth map between two Poisson manifolds
is Poisson if and only if its graph is a coisotropic submanifold of
the product of the two Poisson manifolds. Hence, the problem reduces
to showing that a submanifold $C$ of a Poisson manifold $X$ is coisotropic
if and only if $TC$ is coisotropic in $TX$. 

Recall that the Poisson structure on $TX$ can be described locally
in terms of the matrix
\[
\tilde{\Pi}_{X}(x,v)=\left(\begin{array}{cc}
0 & \Pi_{X}(x)\\
-\Pi_{X}(x) & \partial_{k}\Pi_{X}(x)v^{k}
\end{array}\right)
\]
and that we can identify a vector in $T_{(x,v)}TC$ with $\delta x\oplus\delta v\in T_{x}C\oplus T_{x}C$.
This further gives the identification of $N_{(x,v)}^{*}TC$ with $N_{x}^{*}C\oplus N_{x}^{*}C$.
Now $TC$ is coisotropic iff for all $\theta\oplus\nu\in N_{(x,v)}^{*}TC$,
we have that 
\[
\tilde{\Pi}_{X}^{\sharp}(x,v)(\theta\oplus\nu)=\Pi_{X}^{\sharp}(x)\eta\oplus(-\Pi_{X}^{\sharp}(x)\theta+\partial_{k}\Pi_{X}^{\sharp}(x)v^{k}\eta)\in T_{x}C\oplus T_{x}C,
\]
which is equivalent to $C$ being coisotropic, since $\partial_{k}\Pi_{X}^{\sharp}(x)v^{k}\eta$
is always in $T_{x}C$ provided that $C$ is coisotropic (to see this,
consider the derivative at $0$ of the curve $(x(t),\Pi_{X}(x(t))\eta)$
in $TC$ such that $x(0)=x$, $\dot{x}(0)=v$, and $\eta$ is a section
of $N^{*}C$). 

This shows that $\phi$ is Poisson if and only if $T^{*}\phi$ is
a Lie algebroid comorphism. Let us check now that $\phi$ is complete
whenever $T^{*}\phi$ is. 

First of all, a direct computation shows that the hamiltonian vector
field $\xi_{\phi^{*}f}$ where $f\in C^{\infty}(Y)$ coincides with
$\Pi_{X}^{\sharp}(T^{*}\phi)^{\dagger}df$. Moreover, by definition,
$\xi_{f}$ is complete if and only if the section $df$ is complete. 

Suppose now that the comorphism $(\phi,T^{*}\phi)$ is complete. Take
a complete hamiltonian vector field $\xi_{f}$. Then $(T^{*}\phi)^{\dagger}df$
is complete (since $df$ is complete) which implies thus that $\xi_{\phi^{*}f}$
is also complete. Therefore, the Poisson map $\phi$ is complete.

Conversely, suppose that $\phi$ is complete and let $s\in\Gamma(B)$
be a complete section. We need to show that the integral curve of
$(T^{*}\phi)^{\dagger}s$ through any $x\in X$ exists for all times.
For this, consider the integral curve $y(t)$ of $\Pi_{Y}^{\sharp}s$
passing through $y=\phi(x)$ at time $t_{0}$. Choose a partition
$\R=\cup_{i\in\Z}[t_{i},t_{i+1}]$ such that there exist a family
of open subsets $U_{i}$ and compact subsets $K_{i}$ such that $[t_{i},t_{i+1}]\subset K_{i}\subset U_{i}$
and the one-form $s$ is exact on the $U_{i}$'s (i.e. $s_{|U_{i}}=df_{i}$
for some smooth function $f_{i}$ on $U_{i}$). For each $i\in\Z$,
choose further a cut-off function $\chi_{i}$ that is $1$ on $K_{i}$
and zero outside of $U_{i}$. Since the function $\tilde{f}_{i}:=\chi_{i}f_{i}$
has compact support, its hamiltonian vector field is complete, and
thus $\xi_{\phi^{*}\tilde{f}_{i}}=\Pi_{X}^{\sharp}(T^{*}\phi)^{\dagger}d\tilde{f}_{i}$
is also complete, because $\phi$ is complete. We can now construct
the integral curve of $\Pi_{X}^{\sharp}(T^{*}\phi)^{\dagger}$ for
all times by starting at $x$ and following the flow of $\Pi_{X}^{\sharp}(T^{*}\phi)^{\dagger}d\tilde{f}_{i}$
(which exists for all times) between $t_{i}$ and $t_{i+1}$ for each
$i\in\Z$. 
\end{proof}

The graph of a Poisson map $\phi$ from $X$ to $Y$ is a coisotropic
submanifold of the Poisson manifold product $\overline{X}\times Y$.
Thus, the conormal bundle $N^{*}\graph\phi$ to this graph is a subalgebroid
of the algebroid product $\overline{T^{*}X}\times T^{*}Y$. Proposition
\ref{prop: complete} tells us that this conormal bundle is actually
the graph of the Lie algebroid comorphism $(\phi,T^{*}\phi)$ from
$T^{*}X$ to $T^{*}Y$. Applying the the path construction to this
comorphism yields a hypercomorphism 
\[
\iota:\Sigma(N^{*}\graph\phi)\rightarrow\overline{\Sigma(T^{*}X)}\times\Sigma(T^{*}Y)
\]
between the integrating symplectic groupoids, which happens to be
in general only a lagrangian immersion (see \cite{cattaneo2004} for
instance). 

Using Theorem \ref{thm: Dazord} and Proposition \ref{prop: complete},
we obtain that, for a complete Poisson map, the lagrangian immersion
$\iota$ is a closed lagrangian embedding, and its image, which we
denote by $\Sigma(\phi)$, is a closed lagrangian submanifold of $\overline{\Sigma(T^{*}X)}\times\Sigma(T^{*}Y)$.
In other words, for complete Poisson maps, $\Sigma(\phi)$ is, at
the same time, a canonical relation from $\Sigma(T^{*}X)$ to $\Sigma(T^{*}Y)$,
a lagrangian subgroupoid over the graph of $\phi$, and a comorphism
from $\Sigma(T^{*}X)$ to $\Sigma(T^{*}Y)$. 

In complete analogy with the Lie algebroid case, we can summarize
the discussion above by the following statements, some of which can
already be found in \cite{CF,da:groupoide,zakrzewski1990II}: 

\begin{prop}
(Zakrzewski \cite{zakrzewski1990II}) \label{prop: Poisson 1} Let
$G$ and $H$ be symplectic groupoids over $X$ and $Y$ respectively,
and let $(\phi,\Phi)$ be a symplectic comorphism from $G$ to $H$.
Then $\phi=\Lie(\phi,\Phi)$ is a \textbf{complete} Poisson map from
$X$ to $Y$.
\end{prop}

\begin{proof}
The lagrangian subgroupoid $\gamma_{(\phi,\Phi)}\rightrightarrows\graph\phi$
integrates the coisotropic submanifold $\graph\phi$, and, thus, integrates
the corresponding Lie subalgebroid $N^{*}\graph\phi$ (see \cite{cattaneo2004}
for instance), which is nothing but the graph of the comorphism $T^{*}\phi$
from $T^{*}X$ to $T^{*}Y$. By Proposition \ref{prop: Lie}, we have
then that $T^{*}\phi$ is complete because integrable, and, hence,
that $\phi$ is complete by Proposition \ref{prop: complete}
\end{proof}

\begin{thm}
\label{thm: Poisson 2} The path construction $\Sigma$ is a functor
from the category of integrable Poisson manifolds and complete Poisson
maps to the category of source $1$-connected symplectic groupoids
and symplectic comorphisms. It is an inverse to the Lie functor $\Lie$,
and, thus, implements an equivalence between these two categories. 
\end{thm}

\begin{cor}
(Caseiro-Fernandes \cite{CF}, Dazord \cite{da:groupoide}, Zakrzewski \cite{zakrzewski1990II})
\label{cor: Poisson 3} Let $X$ and $Y$ be two integrable Poisson
manifolds with source $1$-connected integrating symplectic groupoids
$G$ and $H$. Then a Poisson map $\phi$ from $X$ to $Y$ integrates
to a (unique) symplectic comorphism $(\phi,\Phi)$ from $G$ to $H$
if and only if it is complete. 
\end{cor}

Theorem \cite{fernandes2006} provides a rigorous foundation for constructions
in Poisson geometry involving $\Sigma$ in the spirit of Fernandes
in \cite{fernandes2006} where $\Sigma$ (which is called the ``symplectization
functor'') is only considered heuristically. 

\begin{rem}
A weaker version of Corollary \ref{cor: Poisson 3} was already stated
without proof by Dazord in \cite{da:groupoide}, where only the implication
from complete to integrable was considered. Recently, Caseiro and Fernandes in \cite{CF}
proved that the object integrating a complete Poisson map from integrable Poisson manifolds $X$ to $Y$ is
an embedded lagrangian subgroupoid of the symplectic groupoid product
$\overline{\Sigma(X)}\times\Sigma(Y)$, using the path construction and similar lifting
properties as described in Section \ref{sec:lifting}.
However, one can find
versions and proofs of both this Corollary and Proposition \ref{prop: Poisson 1}
in Zakrzewski's paper \cite{zakrzewski1990II}, which was written
even before the notion of comorphisms was formally introduced by Higgins
and Mackenzie in \cite{HM1993}. In \cite{zakrzewski1990II}, Zakrzewski
integrates complete Poisson maps not to symplectic comorphisms but
to what he called ``morphisms of $S^{*}$-algebras,'' which are
defined as special canonical relations satisfying certain algebraic
relations. One can shows that that Zakrzewski's morphisms of $S^{*}$-algebras
are nothing but symplectic comorphisms. His proof relies mostly on
the method of characteristics for coisotropic submanifolds. 
\end{rem}

\end{document}